\documentclass[12pt,leqno]{article}
\usepackage{amssymb,amsmath,amsfonts,amsbsy, amsthm, xspace, latexsym, amscd, graphicx, fancyhdr}
\usepackage{pb-diagram, geometry}

\swapnumbers
\theoremstyle{definition}
\newtheorem{definition}{Definition}[section]
\newtheorem{remark}[definition]{Remark}
\newtheorem{example}[definition]{Example}
\newtheorem{examples}[definition]{Examples}

\theoremstyle{theorem}
\newtheorem{theorem}[definition]{Theorem}

\newtheorem{lemma}[definition]{Lemma}
\newtheorem{corollary}[definition]{Corollary}

\newcommand{\ifff}{{\em i\,f\,{f}\;  }}

\newcommand{\supp}{\ensuremath{{\rm{supp}}}}

\newcommand{\card}{\ensuremath{{\rm{card}}}}

\newcommand{\dom}{\ensuremath{{\rm{dom}}}}

\newcommand{\ran}{\ensuremath{{\rm{ran}}}}

\newcommand{\st}{\ensuremath{{{\rm{st}}}}}

\newcommand{\emb}{\ensuremath{\hookrightarrow}}

\newcommand{\rest}{\ensuremath{\upharpoonright}} 

\newcommand{\conv}{\ensuremath{\circledast}}

\newcommand{\Diff}{\ensuremath{{\rm{Diff}}}}

\begin{document}
\title{Full Algebra of Generalized Functions and Non-Standard Asymptotic Analysis}
\author{Todor D. Todorov*\\ 
                        Mathematics Department\\                
                        California Polytechnic State University\\
                        San Luis Obispo, California 93407, USA\\
																							(ttodorov@calpoly.edu)\\
Hans Vernaeve*\\\
Unit for Engineering Mathematics\\\
University of Innsbruck, A-6020 Innsbruck, Austria\\
(hans.vernaeve@uibk.ac.at)}
\date{}
\maketitle

\begin{abstract} We construct an algebra of generalized functions endowed with a canonical embedding of the
space of Schwartz distributions. We offer a
solution to the problem of multiplication of Schwartz distributions similar to but different from Colombeau's
solution. We show that the set of scalars of our algebra is
an algebraically closed field unlike its counterpart in Colombeau theory, which is a ring with
zero divisors. We prove a Hahn-Banach extension principle which does not
hold in Colombeau theory. We establish a connection between our theory with non-standard analysis and thus
answer, although indirectly, a question raised by J.F. Colombeau. This article provides a bridge between
Colombeau theory of generalized functions and non-standard analysis.
\end{abstract}

\noindent *Both authors were partly supported by START-project Y237 of the Austrian Science Fund.
The second author was also supported by research grant M949 of Austrian Science Fund.


\section{Introduction}\label{S: Introduction}

\quad In the early 70's, A. Robinson introduced a real closed, non-archimedean field 
${^\rho\mathbb{R}}$ (Robinson~\cite{aRob73}) as a factor ring of non-standard numbers  in
$^*\mathbb{R}$ (Robinson~\cite{aRob66}). The field 
${^\rho\mathbb{R}}$ is known as Robinson field of  asymptotic numbers (or Robinson valuation field), because
it is a natural framework of the classical asymptotic analysis  (Lightstone \& Robinson~\cite{LiRob}). 
Later W.A.J. Luxemburg~\cite{wLux} established a connection between $^\rho\mathbb{R}$ and $p$-adic analysis
(see also the beginning of  Section~\ref{S: A Hahn-Banach Like Theorem for
Asymptotic Functionals}  in this article). Li Bang-He~\cite{leBangHe78} studied the connection between  
$^\rho\mathbb{R}$ and the analytic representation of Schwartz distributions, and V. Pestov~\cite{vPes} involved
the field
$^\rho\mathbb{R}$ and similar constructions in the theory of Banach spaces.  More recently, it was shown
that the field $^\rho\mathbb{R}$ is isomorphic to a particular Hahn field of generalized power series
(Todorov \& Wolf~\cite{TodWolf}). The algebras
$^\rho\mathcal{E}(\Omega)$ of
$\rho$-asymptotic functions were introduced in (Oberguggenberger \&
Todorov~\cite{OberTod98}) and studied in Todorov~\cite{tTod99}. It is a differential algebra over
Robinson's field
$^\rho\mathbb{C}$ containing a copy of the Schwartz distributions $\mathcal{D}^\prime(\Omega)$
(Vladimirov~\cite{vVladimirov}). Applications of $^\rho\mathcal{E}(\Omega)$ to partial differential
equations were presented in Oberguggenberger~\cite{mOber95}. We sometimes refer to the
mathematics associated directly or indirectly with the fields
${^\rho\mathbb{R}}$ as {\bf non-standard asymptotic analysis}.  

	\quad On the other hand, in the early 80's,  J.F. Colombeau developed his theory of
{\em new generalized functions} without any connection, at least initially, with non-standard analysis
(Colombeau~\cite{jfCol84a}-\cite{jfCol92}). This theory is known as \textbf{Colombeau theory} or 
\textbf{non-linear theory of generalized functions} because it solves the {\em problem of the
multiplication of Schwartz distributions}.  Here is a \textbf{summary of Colombeau
theory} presented in axiomatic like fashion: Let $\mathcal{T}^d$ denote
the usual topology on $\mathbb{R}^d$ and let
$G$ be an open set of  $\mathbb{R}^d$. A set 
$\mathcal{G}(G)$ is called a \textbf{special} \textbf{algebra of generalized functions on $G$} (of
Colombeau type) if there exists a family
$\mathcal{G}=:\{\mathcal{G}(\Omega)\}_{\Omega\in\mathcal{T}^d}$ (we use $=:$ for ``equal by
definition'') such that:
\begin{enumerate}
\item Each $\mathcal{G}(\Omega)$ is a \textbf{commutative differential ring}, i.e. $\mathcal{G}(\Omega)$ is a
commutative ring supplied with partial derivatives $\partial^\alpha, \alpha\in\mathbb{N}_0^d$ (linear
operators obeying the chain rule). Here $\mathbb{N}_0=\{0, 1, 2,\dots\}$. Let $\overline{\mathbb{C}}$ denote
the \textbf{ring of generalized scalars} of the family 
$\mathcal{G}$ defined as the set of the functions in $\mathcal{G}(\mathbb{R}^d)$ with
zero gradient. Each
$\mathcal{G}(\Omega)$ becomes a
\textbf{differential algebra over the ring}
$\overline{\mathbb{C}}$ (hence, the terminology {\em Colombeau algebras}, for short).  

\item The ring of generalized scalars $\overline{\mathbb{C}}$ is of the form 
$\overline{\mathbb{C}}=\overline{\mathbb{R}}\oplus i\overline{\mathbb{R}}$, where 
$\overline{\mathbb{R}}$ is a partially ordered real ring, which is a proper extension of
$\mathbb{R}$. ({\em Real ring} means a ring with the property that $a_1^2+a_1^2+\dots+a_n^2=0$
implies $a_1=a_2=\dots=a_n=0$). The formula $|x+iy|=\sqrt{x^2+y^2}$ defines an absolute value on $\overline{\mathbb{C}}$. 
Consequently,
$\overline{\mathbb{C}}$ is a proper extension of 
$\mathbb{C}$ and both
$\overline{\mathbb{R}}$ and
$\overline{\mathbb{C}}$  contain non-zero infinitesimals. In Colombeau theory the infinitesimal
relation $\approx$ in $\overline{\mathbb{C}}$ is called \textbf{association}. 
\item   $\overline{\mathbb{C}}$ is
\textbf{spherically complete} under some ultra-metric $d_v$ on $\overline{\mathbb{C}}$ which agrees with
the  partial order in $\overline{\mathbb{R}}$ in the sense that  $|z_1|<|z_2|$ implies $d_v(0, z_1)\leq
d_v(0, z_2)$.
\item For every $f\in\mathcal{G}(\Omega)$ and every test function $\tau\in\mathcal{D}(\Omega)$  a
\textbf{pairing} $\left(f\left.\right|\tau\right)\in\overline{\mathbb{C}}$  is
defined (with the usual linear properties). Here $\mathcal{D}(\Omega)$
stands for the class of $\mathcal{C}^\infty$-functions from $\Omega$ to $\mathbb{C}$ with compact
supports. Let
$f, g\in\mathcal{G}(\Omega)$. The functions
$f$ and $g$ are called \textbf{weakly equal} (or {\em equal in the sense of generalized distributions}), in
symbol
$f\cong g$, if $\left(f\left.\right|\tau\right)=\left(g\left.\right|\tau\right)$ for all
$\tau\in\mathcal{D}(\Omega)$. Similarly,  $f$ and  $g$ are \textbf{weakly associated} (or simply,
{\em associated}, for short), in symbol
$f\approx g$, if
$\left(f\left.\right|\tau\right)\approx\left(g\left.\right|\tau\right)$ for all $\tau\in\mathcal{D}(\Omega)$,
where $\approx$ in the latter formula stands for the infinitesimal relation in $\overline{\mathbb{C}}$.
\item The family $\mathcal{G}$ is a {\bf sheaf}. That means that  $\mathcal{G}$ is
supplied with a restriction $\rest$ to an open set (with the usual sheaf properties, cf.  A.
Kaneko~\cite{aKan88}) such that
$\mathcal{T}^d\ni\mathcal{O}\subseteq\Omega$ and $f\in\mathcal{G}(\Omega)$
implies $f\rest\mathcal{O}\in\mathcal{G}(\mathcal{O})$. Consequently, each generalized function
$f\in\mathcal{G}(\Omega)$ has a {\bf support} $\supp(f)$ which is a closed subset of $\Omega$. 
\item Let  $\Omega, \Omega^\prime\in\mathcal{T}^d$ and $\Diff(\Omega^\prime,
\Omega)$ denote the set of  all 
$\mathcal{C}^\infty$-diffeomorphisms from
$\Omega^\prime$ to
$\Omega$ ($\mathcal{C}^\infty$-bijections with $\mathcal{C}^\infty$-inverse).  A
\textbf{composition} (change of variables) $f\circ\psi\in\mathcal{G}(\Omega^\prime)$ is defined for all
$f\in\mathcal{G}(\Omega)$ and all $\psi\in\Diff(\Omega^\prime,
\Omega)$.
\item For every $\Omega\in\mathcal{T}^d$ there exists an \textbf{embedding}
$E_\Omega:\mathcal{D}^\prime(\Omega)\to\mathcal{G}(\Omega)$  of the space of Schwartz distributions
$\mathcal{D}^\prime(\Omega)$ into  $\mathcal{G}(\Omega)$ such that: 
\begin{description}
\item[(a)]
$E_\Omega$ preserves the vector operations and
partial differentiation in
$\mathcal{D}^\prime(\Omega)$; 
\item[(b)]  $E_\Omega$ is sheaf-preserving, i.e.  $E_\Omega$
preserves the restriction to open sets. Consequently, $E_\Omega$
preserves the support of the Schwartz distributions. 
\item[(c)] $E_\Omega$ preserves the ring operations and partial differentiation in the class
$\mathcal{E}(\Omega)$. Here $\mathcal{E}(\Omega)$ stands for the class of 
$\mathcal{C}^\infty$-functions from $\Omega$ to $\mathbb{C}$ (where $\mathcal{E}(\Omega)$ is treated as a
subspace of $\mathcal{D}^\prime(\Omega)$).

\item[(d)] $E_\Omega$ preserves the pairing between $\mathcal{D}^\prime(\Omega)$ and the class of test
functions $\mathcal{D}(\Omega)$. Consequently, $E_\Omega$ preserves weakly the Schwartz multiplication in
$\mathcal{D}^\prime(\Omega)$ (multiplication by duality). 

\item[(e)]  $E_\Omega$ preserves the usual multiplication in the
class of continuous functions $\mathcal{C}(\Omega)$ up to functions in $\mathcal{G}(\Omega)$ that are weakly
associated to zero. 
\item[(f)]  $E_\Omega$ preserves weakly the composition with
diffeomorphisms (change of variables) in the sense that for every $\Omega, \Omega^\prime\in\mathcal{T}^d$,
every $T\in\mathcal{D}^\prime(\Omega)$ and every
$\psi\in\Diff(\Omega^\prime, \Omega)$ we have  $\left(E_\Omega(T)\circ\psi
\left.\right|\tau\right)=$$\left(E_{\Omega^\prime}(T\circ\psi)
\left.\right|\tau\right)$ for all test functions $\tau\in\mathcal{D}(\Omega^\prime)$. Here $T\circ\psi$
stands for the composition in the sense of the distribution theory (V. Vladimirov~\cite{vVladimirov}).
\end{description}

\item A special algebra is called a \textbf{full algebra of
generalized functions} (of Colombeau type) if the embedding
$E_\Omega$ is \textbf{canonical} in the sense that $E_\Omega$ can be uniquely determined by
properties  expressible only in terms which are  already involved in the definition of the family
$\mathcal{G}=:\{\mathcal{G}(\Omega)\}_{\Omega\in\mathcal{T}^d}$. 
\item A family $\mathcal{G}=\{\mathcal{G}(\Omega)\}_{\Omega\in\mathcal{T}^d}$ of algebras of
generalized functions (special or full)  is called
\textbf{diffeomorphism-invariant} if  $E_\Omega$ preserves the composition with
diffeomorphisms in the sense that  $E_\Omega(T)\circ\psi
=E_{\Omega^\prime}(T\circ\psi)$  for all $\Omega, \Omega^\prime\in\mathcal{T}^d$, all
$T\in\mathcal{D}^\prime(\Omega)$ and all $\psi\in\Diff(\Omega^\prime, \Omega)$.
\end{enumerate} 

	  We should mention that embeddings $E_\Omega$ (canonical or not) of the type described above are, in
a sense, {\em optimal} in view of the restriction imposed by the Schwartz impossibility results
(Schwartz~\cite{lSchwartz54}). For a discussion on the topic we refer to (Colombeau~\cite{jfCol92},
p. 8). Every family of algebras 
$\mathcal{G}(\Omega)$ (special or full) of the type described above offers a
\textbf{solution to the problem of the multiplication of Schwartz distributions} because the Schwartz distributions
can be multiplied within an associative and commutative differential algebra.

		 Full algebras of generalized functions were constructed first by J. F. Colombeau~\cite{jfCol84a}. Several
years later, in an attempt to simplify  Colombeau's original construction
J.F. Colombeau and A.Y. Le Roux~\cite{ColRoux88}  (and other authors, H. A. Biagioni~\cite{aBiag90}) defined
the so called {\em simple algebras} of generalized functions. Later M. Oberguggenberger~(\cite{mOber92},
Ch.III, \S 9) proved that the simple algebras are, actually, special algebras  in the sense explained above.
Diffeomorphism invariant full algebras were developed in (Grosser, Kunzinger, Oberguggenberger \&
Steinbauer~\cite{mGrosser at al 1}-\cite{mGrosser at al 2}). The sets of generalized scalars of all these algebras
are rings with zero divisors (Colombeau~\cite{jfCol84a}, pp. 136). The algebras of
$\rho$-asymptotic functions $^\rho\mathcal{E}(\Omega)$ \cite{OberTod98}, mentioned earlier, are special
algebras of Colombeau type with set of generalized scalars which is an
algebraically closed field. The counterpart of the embedding $E_\Omega$ in \cite{OberTod98} is denoted by
$\Sigma_{D,\Omega}$. It is certainly not canonical because the existence of 
$\Sigma_{D,\Omega}$ is proved in \cite{OberTod98} by saturation principle (in a non-standard analysis
framework) and then ``fixed by hand'' (see Remark~\ref{R: Non-Canonical Embedding}). Among other
things the purpose of this article is to construct a canonical embedding  $E_\Omega$ in
$^\rho\mathcal{E}(\Omega)$. We achieve this by means of the choice of a particular ultra-power
non-standard model (Section~\ref{S: Distributional Non-Standard Model}) and a particular choice of the
positive infinitesimal  $\rho$ within this model (Definition~\ref{D: Distributional Non-Standard
Model}, \#\ref{No: Canonical Infinitesimal}).

	 Colombeau theory has numerous applications to ordinary and partial differential equations,  the
theory of elasticity, fluid mechanics, theory of shock waves (Colombeau\cite{jfCol84a}-\cite{jfCol92}, 
Oberguggenberger~\cite{mOber92}), to differential geometry and relativity theory (Grosser, Kunzinger,
Oberguggenberger \& Steinbauer~\cite{mGrosser at al 2}) and, more recently, to quantum field
theory (Colombeau, Gsponer \& Perrot~\cite{jfCol07}).

	 Despite the remarkable achievement and promising applications  the theory of Colombeau 
has some features which can be certainly improved. Here are some of them: 
\begin{description}

	\item{(a)} The ring of generalized scalars
$\overline{\mathbb{C}}$ and  the algebras of generalized  functions $\mathcal{G}(\Omega)$ in Colombeau
theory are constructed as factor rings within   the ultrapowers $\mathbb{C}^I$ and
$\mathcal{E}(\Omega)^I$, respectively, for a particular  index set $I$.
The rings of nets such as $\mathbb{C}^I$ and $\mathcal{E}(\Omega)^I$ however (as well their subrings) lack
general theoretical principles similar to the axioms of $\mathbb{R}$ and $\mathbb{C}$, for example.
Neither $\mathbb{C}^I$ and
$\mathcal{E}(\Omega)^I$ are endowed with principles such as the transfer
principle or internal definition principle in non-standard analysis.  For that reason
Colombeau theory has not been able so far to get rid of the index set $I$ even after the factorization
which transforms $\mathbb{C}^I$ and $\mathcal{E}(\Omega)^I$ into $\overline{\mathbb{C}}$ and
$\mathcal{G}(\Omega)$, respectively. As a result Colombeau theory remains overly constructive: there
are too many technical parameters (with origin in the index set
$I$) and too many quantifiers in the definitions and theorems. To a certain extent, Colombeau theory
resembles what would be the real analysis if it was based not on the axioms of the
reals $\mathbb{R}$ but rather on Cauchy's construction of the real numbers as equivalence classes of
fundamental sequences in $\mathbb{Q}$.

\item{(b)} In a recent article  M. Oberguggenberger and H. Vernaeve~\cite{OberVernaeve07} defined the
concept of {\em internal sets} of  $\overline{\mathbb{C}}$ and $\mathcal{G}(\Omega)$ and showed that
theoretical principles similar to order completeness, underflow and overflow principles and
saturation principle for internal sets of
$\overline{\mathbb{C}}$ and $\mathcal{G}(\Omega)$ hold in Colombeau theory as well although  in
more restrictive sense compared with non-standard analysis.  However the sets of generalized scalars for
$\overline{\mathbb{R}}$ and
$\overline{\mathbb{C}}$ are still rings with zero divisors and  $\overline{\mathbb{R}}$ is only a partially
ordered (not totally ordered) ring. These facts lead to technical complications. For example 
Hahn-Banach extension principles do not hold in Colombeau theory (Vernaeve~\cite{hVernaeve}).
\end{description}

	In \textbf{this article}:  
\begin{description}
\item[(i)]  We construct a family of algebras of
generalized functions
$\widehat{\mathcal{E}(\Omega)^{\mathcal{D}_0}}$ called {\em asymptotic functions}
(Section~\ref{S: Asymptotic Numbers and Asymptotic Functions}).   We show that
$\widehat{\mathcal{E}(\Omega)^{\mathcal{D}_0}}$ are full algebras of Colombeau type 
(Section~\ref{S: A Solution to the Problem
of Multiplication of Schwartz Distributions}) in the sense explained above. Thus we offer a
solution to the problem of the multiplication of Schwartz distributions similar to but different 
from Colombeau's solution (Colombeau~\cite{jfCol84a}). Since the full algebras
are commonly considered to be more naturally connected to the theory of Schwartz distributions than
the special algebras, we look upon 
$\widehat{\mathcal{E}(\Omega)^{\mathcal{D}_0}}$ as an \textbf{improved alternative to the algebra of
$\rho$-asymptotic functions}
$^\rho\mathcal{E}(\Omega)$ defined in \cite{OberTod98}. 
\item[(ii)] We believe that our theory is a \textbf{modified and improved alternative to the original Colombeau
theory} for the following reasons: (a) The set of scalars $\widehat{\mathbb{C}^{\mathcal{D}_0}}$ of the algebra
$\widehat{\mathcal{E}(\Omega)^{\mathcal{D}_0}}$, called here {\em asymptotic numbers}, is an algebraically
closed field  (Theorem~\ref{T: Algebraic Properties}). Recall for comparison that
its counterpart in Colombeau theory  $\overline{\mathbb{C}}$ is a ring with zero divisors
(Colombeau~\cite{jfCol84a}, pp. 136).  (b) As a consequence we show that a Hahn-Banach extension
principle holds for linear functionals with values in
$\widehat{\mathbb{C}^{\mathcal{D}_0}}$ (Section~\ref{S: A Hahn-Banach Like
Theorem for Asymptotic Functionals}). This result does not have a counterpart in Colombeau theory
(Vernaeve~\cite{hVernaeve}). (c) At this stage the construction of
$\widehat{\mathcal{E}(\Omega)^{\mathcal{D}_0}}$ is already simpler than its counterpart in
Colombeau~\cite{jfCol84a}; our theory has one (regularization) parameter less.

\item[(iii)] Our next goal is to simplify our theory even more by establishing a connection with non-standard
analysis (Section~\ref{S: J.F. Colombeau's Non-Linear Theory of Generalized Functions and Non-Standard
Analysis}). For this purpose we construct a particular ultrapower
non-standard model called in this article the {\em distributional non-standard model} (Section~\ref{S:
Distributional Non-Standard Model}). Then we replace the rings of nets $\mathbb{C}^I$ and
$\mathcal{E}(\Omega)^I$ in Colombeau theory by the non-standard $^*\mathbb{C}$ and
$^*\mathcal{E}(\Omega)$, respectively and the regularization parameter $\varepsilon$ in Colombeau theory
by a particular (canonical) infinitesimal $\rho$ in $^*\mathbb{R}$. We show that  the field of asymptotic numbers
$\widehat{\mathbb{C}^{\mathcal{D}_0}}$ (defined in Section~\ref{S: Asymptotic Numbers and
Asymptotic Functions}) is isomorphic to a particular Robinson field ${^\rho\mathbb{C}}$
(Robinson~\cite{aRob73}). We also prove that the algebra of asymptotic
functions
$\widehat{\mathcal{E}(\Omega)^{\mathcal{D}_0}}$ (defined in Section~\ref{S: Asymptotic Numbers and
Asymptotic Functions}) is
isomorphic to a particular algebra of
$\rho$-asymptotic functions $^\rho\mathcal{E}(\Omega)$ introduced in (Oberguggenberger \&
Todorov\cite{OberTod98}) in the framework of non-standard analysis.
\item[(iv)] Among other things  this article provides a bridge between Colombeau theory of
generalized functions and non-standard analysis and we hope that it will be beneficial for both.
After all Robinson's non-standard analysis (Robinson~\cite{aRob66}) is historically at least several
decades older than Colombeau theory. A lot of work had been already done in the non-standard setting
on topics similar to those which appear in Colombeau theory. By establishing a connection with
non-standard analysis we answer, although indirectly, a question raised by J.F. Colombeau himself in one
of his ``research projects'' (Colombeau~\cite{jfCol92}, pp. 5).
\end{description}

	  Since the article establishes a connection between two different fields of mathematics, it is written mostly
with  two types of readers in mind.The readers with background in non-standard analysis might find in
Section~\ref{S: Ultrafilter on Test Functions}-\ref{S: A Solution to the Problem of Multiplication of Schwartz
Distributions} and Section~\ref{S: A Hahn-Banach Like Theorem for Asymptotic Functionals} (along with the
axiomatic summary of Colombeau theory presented above) a short introduction to the non-linear theory of
generalized functions. Notice however that in these sections we do not present the original Colombeau theory but
rather a modified (and improved) version of this theory. The reader without background in non-standard
analysis will find in Section~\ref{S: Distributional Non-Standard Model} a short introduction to the subject. The
reading of Sections~\ref{S: Ultrafilter on Test Functions}-\ref{S: A Solution to the Problem of Multiplication of
Schwartz Distributions} does not require background in non-standard analysis. 

\section{Ultrafilter on Test Functions}\label{S: Ultrafilter on Test Functions}

	\quad In this section we define a particular ultrafilter on the class of
test functions $\mathcal{D}(\mathbb{R}^d)$ closely related to Colombeau theory of
generalized functions (Colombeau~\cite{jfCol84a}). We shall often use the \textbf{shorter
notation
$\mathcal{D}_0$ instead of} $\mathcal{D}(\mathbb{R}^d)$. 

	\quad In what follows we denote by $R_\varphi$ the \textbf{radius of support} of
$\varphi\in\mathcal{D}(\mathbb{R}^d)$ defined by
\begin{equation}\label{E: RadiusSupport}
R_\varphi=
\begin{cases}
 \sup\{||x|| : x\in\mathbb{R}^d,\;  \varphi(x)\not= 0 \}, &\varphi\not=0,\\
1, 																																																										&\varphi=0. 
\end{cases}
\end{equation}

\begin{definition}[Directing Sets]\label{D: Directing Sets} We
define the directing sequence of sets $\mathcal{D}_0, \mathcal{D}_1, \mathcal{D}_2\dots$ by letting
$\mathcal{D}_0=\mathcal{D}(\mathbb{R}^d)$ and

\begin{align}
{\mathcal D}_{n}= \Big\{&
\varphi\in{\mathcal D}(\mathbb{R}^d):\notag\\
&\varphi \text{ is real-valued},\notag\\
&(\forall x\in\mathbb{R}^d)(\varphi(-x)=\varphi(x)),\notag\\
&R_\varphi\leq 1/n,\notag\\ 
&\int_{\mathbb{R}^d}\varphi(x)\, dx=1,\notag\\
&(\forall\alpha\in\mathbb{N}_0
^d)\left(1 \leq |\alpha|\leq n \Rightarrow\int_{\mathbb{R}^d}x^\alpha\varphi(x)\, dx=0\right),\notag\\ 
&\int_{\mathbb{R}^d}\left|\varphi(x)\right|\, dx \leq 1 + \frac{1}{n},\notag\\
&(\forall\alpha\in\mathbb{N}_0^d)\left( |\alpha|\leq n \Rightarrow\sup_{x\in\mathbb{R}^d}|{\partial^\alpha
\varphi(x)|}\leq (R_\varphi)^{-2(|\alpha| +d)}\right)\notag\; \Big\},\; n=1, 2, \dots.
\end{align}  
\end{definition}
\begin{theorem}[Base for a Filter]\label{T: Directing Sequence of D(Rd)} 
The directing sequence $(\mathcal{D}_n)$ is a base for a free filter on
$\mathcal{D}_0$ in the sense that
\begin{description}
\item[(i)] $\mathcal{D}(\mathbb{R}^d)=\mathcal{D}_0\supseteq \mathcal{D}_1\supseteq
\mathcal{D}_2\supseteq
\mathcal{D}_3\supseteq\dots$.
\item[(ii)] $\mathcal{D}_n\not=\varnothing$ for all $n\in\mathbb{N}$.
\item[(iii)] $\bigcap_{n=0}^\infty \mathcal{D}_n=\varnothing$.
\end{description}
\end{theorem}

\begin{proof}
  (i) Clear.

  (ii)  Let
$\varphi_0\in\mathcal{D}(\mathbb{R})$ be the test function
\[
\varphi_0(x)=
\begin{cases}
\frac{1}{c}\exp({-\frac{1}{1-x^2}}),&-1\leq x\leq 1,\\
0,&\text{otherwise},
\end{cases}
\]
where $c=\int_{-1}^1 \exp({-\frac{1}{1-x^2}})\,dx$. We let
$C_k=:\sup_{x\in\mathbb{R}}\left|\frac{d^k}{dx^k}\varphi_0(x)\right|$ for each $k\in\mathbb{N}_0$  and
also  $C_\alpha=C_{\alpha_1}\cdots C_{\alpha_d}$ for each multi-index $\alpha\in\mathbb{N}_0^d$. 
For each $n,m\in\mathbb{N}$ we let
\begin{align}
\mathcal{B}_{n,m,d}= \Big\{&\varphi\in\mathcal{D}(\mathbb{R}^d): \notag\\
&\varphi \text{ is real-valued},\notag\\
&\varphi(-x)=\varphi(x) \text{ for all } x\in\mathbb{R}^d,\notag\\
&R_\varphi=\sqrt d,\notag\\
&\int_{\mathbb{R}^d}\varphi(x)\, dx=1,\notag\\
&\int_{\mathbb{R}^d} x^\alpha\varphi(x)\,dx=0\mbox{ for all }\alpha\in\mathbb{N}^d\text{ with } 1\leq
|\alpha|\leq n,\notag\\ 
&\int_{\mathbb{R}^d} |\varphi(x)|\, dx\leq \exp\left(\frac{3d}{m-1}\right),\notag\\
&\sup_{x\in\mathbb{R}^d}\left|\partial^\alpha \varphi(x)\right|\leq C_\alpha (2^d
m^{|\alpha|+d})^{n} \mbox{ for all }\alpha\in\mathbb{N}_0^d\, \Big\}.\notag
\end{align}

	 {\bf Step 1.} We show that, if $m>2$, then $\mathcal {B}_{n,m,d}\not= \varnothing$.
Let first $d=1$. Then $\varphi_0\in\mathcal{B}_{0,m,1}$. By induction on $n$, let
$\varphi_{n-1}\in\mathcal {B}_{n-1,m,1}$. Define $\varphi_n(x) = a\varphi_{n-1}(x)+
b\varphi_{n-1}(mx)$, for some constants $a, b\in\mathbb{R}$ to be determined. Then
\[
\int_{\mathbb{R}}\varphi_n(x)\,dx = a + \frac{b}{m}\quad \text{and} \quad\int_{\mathbb{R}}
x^n\varphi_n(x)\,dx=\left(a+\frac{b}{m^{n+1}}\right)\int_{\mathbb{R}} x^n\varphi_{n-1}(x)\,dx.
\]
To ensure that $\varphi_n\in\mathcal B_{n,m,1}$, we choose $a + \frac{b}{m} = 1$ and
$a+\frac{b}{m^{n+1}}=0$. Solving for $a$, $b$, we find that $a=-\frac{1}{m^n-1}<0$ and
$b=\frac{m^{n+1}}{m^n-1}>0$. Since $a\not= 0$, also $R_{\varphi_n}=1$. Further, since
$\frac{1+x}{1-x}\leq 1 + 3x\leq \exp(3x)$ if $0\le x\le \frac{1}{3}$, we have
\begin{align}
&\int_{\mathbb{R}}\left|\varphi_n(x)\right|\, dx  \leq
\left(|a|+\frac{|b|}{m}\right)\int_{\mathbb{R}}\left|\varphi_{n-1}(x)\right|\, dx=\notag\\  
&\frac{m^n +1}{m^n -1}\int_{\mathbb{R}}\left|\varphi_{n-1}(x)\right|\,  dx\leq
\exp\left(\frac{3}{m^n}\right)\int_{\mathbb{R}}\left|\varphi_{n-1}(x)\right|\, dx,\notag
\end{align}
so inductively,
\[
\int_{\mathbb{R}}\left|\varphi_n(x)\right|\,  dx\leq\prod_{j=1}^n\exp\left(\frac{3}{
m^j}\right)\int_{\mathbb{R}}\left|\varphi_0(x)\right|\, dx\leq
\exp\Big(\sum_{j=1}^\infty\frac{3}{m^j}\Big) =\exp\Big({\frac{3}{m-1}}\Big).
\]
Further, $|a| + |b| m^k= \frac{m^{n+k+1}+1}{m^n-1}\leq 2m^{k+1}$ for $k\geq 0$,  $m> 2$ and $n\geq 1$.
Thus we have
\begin{align}
\sup_{x\in\mathbb{R}}\left|\frac{d^k}{dx^k}\varphi_n(x)\right|
&\leq (|a|+|b|m^k)\sup_{x\in\mathbb{R}}\left|\frac{d^k}{dx^k}\varphi_{n-1}(x)\right|\notag\\
&\le 2m^{k+1} C_k (2 m^{k+1})^{n-1}=C_k (2 m^{k+1})^n.\notag
\end{align}
Hence $\varphi_n\in{\mathcal B}_{n,m,1}$.
Now let $d\in\mathbb{N}$ and $\varphi\in\mathcal {B}_{n,m,1}$ arbitrary. We have
$\psi(x)=:\varphi(x_1)\cdots\varphi(x_d)\in\mathcal{B}_{n,m,d}$.

 {\bf Step 2.} Fix $d\in\mathbb{N}$. Let $n\geq 1$, let $M=\max\{1, \max_{|\alpha|\leq n}C_\alpha\}$,
let
$\psi\in\mathcal{B}_{n,9dn,d}$, let
$\varepsilon=\frac{1}{dM(18dn)^{dn}}$ and let
$\varphi(x)=\frac{1}{\varepsilon^d}\psi(x/\varepsilon)$. We show that
$\varphi\in \mathcal{D}_n$. If $||x||\geq 1/n\geq\varepsilon\sqrt{d}$, then $\varphi(x)=0$. Further, since
$\exp(x)\leq \frac{1}{1-x}$ if $0\leq x < 1$, we have
\[
\int_{\mathbb{R}^d}\left|\varphi(x)\right|\, dx=\int_{\mathbb{R}^d}\left|\psi(x)\right|\,
dx\leq\exp{\left(\frac{3d}{9dn-1}\right)}
\leq 1 + \frac{3d}{9dn-1-3d}\leq 1 + \frac{1}{n}.
\]
Finally, notice that  $R_\varphi=\varepsilon R_\psi=\varepsilon\sqrt{d}$. Thus for $|\alpha|\leq n$ we have
\begin{align}
\sup_{x\in\mathbb{R}^d}\left|\partial^\alpha \varphi(x)\right| 
&\leq
\varepsilon^{-|\alpha|-d}\sup_{x\in\mathbb{R}^d}\left|\partial^\alpha\psi(x)\right|
\leq \varepsilon^{-|\alpha|-d} C_\alpha(2^d(9dn)^{|\alpha| +d})^n\leq\notag\\
&\leq \varepsilon^{-|\alpha|-d} C_\alpha (18dn)^{dn(|\alpha|+d)}
=\varepsilon^{-|\alpha|-d} C_\alpha (\varepsilon dM)^{-|\alpha|-d}\leq\notag\\
&\leq C_\alpha M^{-1}(R_\varphi)^{-2(|\alpha| + d)}.\notag
\end{align}
Hence $\varphi\in\mathcal{D}_n$ as required.

		 (iii) Suppose (on the contrary) that there exists $\varphi\in\bigcap_{n=1}^\infty\,
\mathcal{D}_n$.  That means
(among other things) that $\int_{\mathbb{R}^d}\varphi(x)x^\alpha dx=0$ for all $\alpha\not=0$.
Thus we have $\partial^\alpha\widehat{\varphi}(0)=0$ for all  $\alpha\not=0$, where $\widehat{\varphi}$ denotes the 
Fourier transform of $\varphi$. It follows that $\widehat{\varphi}=C$ for some
constant $C\in\mathbb{C}$ since $\widehat{\varphi}$ is an  entire function on $\mathbb{C}^d$ by the
Paley-Wiener Theorem (Bremermann~\cite{hBremermann65},  Theorem 8.28, pp. 97). Hence by Fourier
inversion, $\varphi=(2\pi)^dC\delta\in\mathcal{D}(\mathbb{R}^d)$, where $\delta$ stands for the Dirac
delta function. The latter implies $C=0$, thus $\varphi=0$,
contradicting the  property $\int_{\mathbb{R}^d}\varphi(x)dx=1$ in the definition of $\mathcal{D}_n$. 
\end{proof}

	 In what follows 
$\mathfrak{c}=:\card{\mathbb{(R)}}$ and $\mathfrak{c}^+$ stands for the successor of
$\mathfrak{c}$. 

\begin{theorem}[Existence of Ultrafilter]\label{T: Existence of Ultrafilter}
There exists a $\mathfrak{c}^+$-good ultrafilter (maximal filter) $\mathcal{U}$ on
$\mathcal{D}_0=:\mathcal{D}(\mathbb{R}^d)$ such that
$\mathcal{D}_n\in\mathcal{U}$ for all $n\in\mathbb{N}_0$ 
(Definition~\ref{D: Directing Sets}).

\end{theorem}
\begin{proof}  We observe  that
$\card{(\mathcal{D}_0)}=\mathfrak{c}$. The existence of a (free) ultrafilter containing
all
$\mathcal{D}_n$ follows easily by Zorn's lemma since the set
$\mathcal{F}=\{A\in\mathcal{P}(\mathcal{D}_0) :  \mathcal{D}_n\subseteq A
\mbox{ for some } n\in\mathbb{N}_0  \}$ is clearly a free filter on $\mathcal{D}_0$. Here
$\mathcal{P}(\mathcal{D}_0)$ stands for the power set of $\mathcal{D}_0$. For the existence of a
$\mathfrak{c}^+$-good ultrafilter containing $\mathcal{F}$ we refer the reader to 
(Chang \& Keisler~\cite{CKeis}) (for a presentation we also mention the  Appendix in Lindstr\o
m~\cite{tLin}).
\end{proof} 

	 Let $\mathcal{U}$ be a $\mathfrak{c}^+$-good ultrafilter on
$\mathcal{D}_0=:\mathcal{D}(\mathbb{R}^d)$ containing all
$\mathcal{D}_n$. We shall \textbf{keep $\mathcal{U}$ fixed} to the end of this article.

	 For those readers who are unfamiliar with the used terminology we present a list
of the most important properties of
$\mathcal{U}$. The properties (1)-(3) below express the fact that $\mathcal{U}$ is a \textbf{filter},  the
property (1)-(4) express the fact that
$\mathcal{U}$ is a
\textbf{free filter}, the property (1)-(5) means that $\mathcal{U}$ is a \textbf{free ultrafilter} (maximal filter)
and (6) expresses the property of $\mathcal{U}$  to be \textbf{$c^+$-good}.

	\begin{lemma}[List of Properties of $\mathcal{U}$]\label{L: List of Properties of U} The ultrafilter
$\mathcal{U}$ is a set of subsets of $\mathcal{D}_0=\mathcal{D}(\mathbb{R}^d)$ such that 
$\mathcal{D}_n\in\mathcal{U}$ for all $n\in\mathbb{N}_0$ and such that: 
\begin{enumerate} 
\item\label{No: F1} If  $A\in\mathcal{U}$ and $B\subseteq \mathcal{D}_0$, then $A\subseteq B$ implies
$B\in\mathcal{U}$. 

\item\label{No: F2} $\mathcal{U}$ is \textbf{closed under finite intersections}. 

\item\label{No: F3} $\varnothing\notin \mathcal{U}$. 
 \item\label{No: Free} $\mathcal{U}$ is a
\textbf{free} filter in the sense that $\cap_{A\in\mathcal{U}}\, A=\varnothing$.

\item\label{No: Ultra} Let $A_k\in\mathcal{P}(\mathcal{D}_0),\, k=1, 2,\dots, n$, for some
$n\in\mathbb{N}$. Then
$\cup_{k=1}^n A_k\in\mathcal{U}$ implies $A_k\in\mathcal{U}$ for at least one $k$. Moreover, 
if the sets $A_k$ are mutually disjoint, then $\cup_{k=1}^n
A_k\in\mathcal{U}$ implies $A_k\in\mathcal{U}$ for exactly one $k$. In particular, for every set
$A\in\mathcal{P}(\mathcal{D}_0)$ exactly one of $A\in\mathcal{U}$ or
$\mathcal{D}_0\setminus A\in\mathcal{U}$ is true. Consequently,  $\mathcal{U}$ has the
\textbf{finite intersection property} (f.i.p.) since $F\notin\mathcal{U}$ and
$\mathcal{D}_0\setminus F\in\mathcal{U}$ for every finite set $F$ of 
$\mathcal{D}_0$.
\item\label{No: Good} $\mathcal{U}$ is \textbf{$\mathfrak{c}^+$-good} in the sense that for every set
$\Gamma\subseteq
\mathcal{D}_0$, with
$\card(\Gamma)\leq \mathfrak{c}$, and every reversal $R:
\mathcal{P}_\omega(\Gamma)\to\mathcal{U}$ there exists a strict reversal
$S: \mathcal{P}_\omega(\Gamma)\to\mathcal{U}$ such that $S(X)\subseteq R(X)$ for all 
$X\in\mathcal{P}_\omega(\Gamma)$. Here 
$\mathcal{P}_\omega(\Gamma)$ denotes the set of all  finite subsets of $\Gamma$. 
\end{enumerate}
\end{lemma}
	
	 Recall that a function $R: \mathcal{P}_\omega(\Gamma)\to\mathcal{U}$ is called a \textbf{reversal}  if
$X\subseteq Y$ implies $R(X)\supseteq R(Y)$ for every $X, Y\in\mathcal{P}_\omega(\Gamma)$. A
\textbf{strict reversal} is a function  $S: \mathcal{P}_\omega(\Gamma)\to\mathcal{U}$ such that $S(X\cup
Y)= S(X)\cap S(Y)$ for every
$X, Y\in\mathcal{P}_\omega(\Gamma)$. 	It is clear that every strict reversal is a reversal (which justifies the
terminology). 
\begin{definition}[Almost Everywhere] \label{D: Almost Everywhere}  Let $P(x)$ be a predicate in one
variable defined on
$\mathcal{D}_0$ (expressing some property of the test functions). We say that $P(\varphi)$ holds almost
everywhere in $\mathcal{D}_0$ or, simply,
$P(\varphi)$ a.e. (where a.e. stands for ``almost everywhere''), if  $\{\varphi\in\mathcal{D}_0 :
P(\varphi)\}\in\mathcal{U}$.
\end{definition}

\begin{example}[Radius of Support]\label{Ex: Radius of Support} Let $R_\varphi$ be the support of
$\varphi$ (cf. (\ref{E: RadiusSupport})) and let $n\in\mathbb{N}$. Then $(R_\varphi\in\mathbb{R}_+\,
\&\, R_\varphi<1/n)$ a.e. because $\mathcal{D}_n\subseteq\{\varphi\in
\mathcal{D}_0 : R_\varphi\in\mathbb{R}_+\, \&\,
R_\varphi<1/n\}$ implies $\{\varphi\in
\mathcal{D}_0 : R_\varphi\in\mathbb{R}_+\, \&\,
R_\varphi<1/n\}\in\mathcal{U}$ by \#\ref{No: F1} of Lemma~\ref{L: List of Properties of U}.
\end{example}

	 The justification of  the terminology ``almost everywhere'' is based on the observation that  the mapping
$\mathcal{M}_\mathcal{U}:\mathcal{P}(\mathcal{D}_0)\to \{0,1\}$,
defined by $\mathcal{M}_\mathcal{U}(A)=1$ if $A\in\mathcal{U}$ and 
$\mathcal{M}_\mathcal{U}(A)=0$ if $A\notin\mathcal{U}$  is finitely additive probability measure on
$\mathcal{D}_0$. 
\section{$\mathcal{D}_0$-Nets and Schwartz Distributions}\label{S: D0-Nets and Schwartz
Distributions}
\begin{definition}[Index Set and Nets]\label{D: Index Set and Nets}  Let $\mathcal{D}_0, \mathcal{D}_1,
\mathcal{D}_2,\dots$ be the directing sequence defined in (Definition~\ref{D: Directing Sets}), where
$\mathcal{D}_0=\mathcal{D}(\mathbb{R}^d)$. Let $S$ be a set. The functions of the form $A:
\mathcal{D}_0\to S$ are called  \textbf{$\mathcal{D}_0$-nets} in
$S$  or, simply \textbf{nets in $S$} for short (Kelley~\cite{jKelley}, p. 65). We denote by $S^{\mathcal{D}_0}$ the set of all 
$\mathcal{D}_0$-nets in $S$. The space of test functions $\mathcal{D}_0$ is the
\textbf{index set} of the nets. If  $A\in S^{\mathcal{D}_0}$ is a net in $S$, we shall often write $A_\varphi$ and
$(A_\varphi)$ instead of $A(\varphi)$ and $A$, respectively. 
\end{definition}

		 In this section we present several technical lemmas about  $\mathcal{D}_0$-nets which are closely
related to the Schwartz theory of distributions and the directing sequence $(\mathcal{D}_n)$ (Section~\ref{S:
Ultrafilter on Test Functions}). Our terminology and notation in distribution theory is close to those in
Vladimirov~\cite{vVladimirov}. We start with several examples of  $\mathcal{D}_0$-nets.
\begin{examples}[Nets and Distributions]\label{Exs: Nets and Distributions} 
\begin{enumerate} 
\item\label{No: Nets in C} We denote by
$\mathbb{C}^{\mathcal{D}_0}$  the set of all
nets of the form $A: \mathcal{D}_0\to\mathbb{C}$. We shall
often write $(A_\varphi)$ instead of $A$ for the nets in
$\mathbb{C}^{\mathcal{D}_0}$. It is clear that
$\mathbb{C}^{\mathcal{D}_0}$ is a ring with zero divisors under the usual pointwise operations. Notice
that the nets in $\mathbb{C}^{\mathcal{D}_0}$ can be viewed as \textbf{complex valued functionals}
(not necessarily linear) on the space of test functions
$\mathcal{D}(\mathbb{R}^d)$. 
\item\label{No: Nets in E(Omega)}   Let $\Omega$ be an open subset of $\mathbb{R}^d$ and
$\mathcal{E}(\Omega)=:\mathcal{C}^\infty(\Omega)$. We denote by
$\mathcal{E}(\Omega)^{\mathcal{D}_0}$ the set of all nets of the form $f:
\mathcal{D}_0\to\mathcal{E}(\Omega)$. We shall often write $(f_\varphi)$ or $(f_\varphi(x))$  instead of $f$
for the nets in
$\mathcal{E}(\Omega)^{\mathcal{D}_0}$.

\item\label{No: Cut-off Net}  Let  $S$ be a set and  $\mathcal{P}(S)$ stand for the power set of $S$. We denote
by $\mathcal{P}(S)^{\mathcal{D}_0}$ the set of all nets of the form $\mathcal{A}:
\mathcal{D}_0\to\mathcal{P}(S)$. We shall often write $(\mathcal{A}_\varphi)$ instead of
$\mathcal{A}$ for the nets in $\mathcal{P}(S)^{\mathcal{D}_0}$.

\item\label{No: Cut-Off} Let $\mathcal{T}^d$ denote the \textbf{usual topology on} $\mathbb{R}^d$. For
every open set 
$\Omega\in\mathcal{T}^d$ we let
\begin{align}
&\Omega_\varphi=\left\{x\in\Omega\mid d(x,\partial\Omega)>R_\varphi \right\},\notag\\
&\widetilde{\Omega}_\varphi=\left\{x\in\Omega\mid d(x,\partial\Omega)>2R_\varphi \; \&\;
||x||<1/R_\varphi 
\right\},\notag
\end{align}
where $d(x,\partial\Omega)$ stands for the Euclidean distance between $x$ and the boundary
$\partial\Omega$ of $\Omega$ and $R_\varphi$ is defined by (\ref{E: RadiusSupport}). Let
$\chi_{\Omega,\varphi}:
\mathbb{R}^d\to
\mathbb{R}$ be the characteristic function of the set
$\widetilde{\Omega}_\varphi$. The \textbf{cut-off net}
$(C_{\Omega,\varphi})\in\mathcal{E}(\mathbb{R}^d)^{\mathcal{D}_0}$ associated with $\Omega$ is defined
by the formula 
$C_{\Omega,\varphi}=: \chi_{\Omega,\varphi}\star\varphi$, where  $\star$ stands for the usual
\textbf{convolution}, i.e. 
\[
C_{\Omega,\varphi}(x)=\int_{\widetilde{\Omega}_\varphi}\, \varphi(x-t)\, dt,
\]
for all $x\in\mathbb{R}^d$ and all $\varphi\in\mathcal{D}_0$. Notice that
$\supp(C_{\Omega,\varphi})\subseteq\Omega_\varphi$ (Vladimirov~\cite{vVladimirov},
Ch.I,
\S4, 6.T).
\item\label{No: phi-regularization} Let $T\in\mathcal{D}^\prime(\Omega)$ be a \textbf{Schwartz
distribution}  on
$\Omega$.  The $\varphi$-\textbf{regularization of}
$T$ is the net $(T_\varphi)\in\mathcal{E}(\Omega)^{\mathcal{D}_0}$ defined by the formula $T_\varphi=:
T\circledast\varphi$, where \textbf{$T\circledast\varphi$ is a short notation for}
$(C_{\Omega,\varphi}T)\star\varphi$  and  $\star$ stands (as before) for
the \textbf{usual convolution}. In other words, we have
\[
T_\varphi(x)=\left(T(t)\left.\right|C_{\Omega,\varphi}(t)\varphi(x-t)\right),
\]
for all $x\in\Omega$ and all $\varphi\in\mathcal{D}_0$. Here $(\,\cdot\,|\,\cdot\,)$ stands for the pairing
between  $\mathcal{D}^\prime(\Omega)$ and $\mathcal{D}(\Omega)$
(Vladimirov~\cite{vVladimirov}). 

\item\label{No: Loc} We denote by $L_\Omega:
\mathcal{L}_{loc}(\Omega)\to\mathcal{D}^\prime(\Omega)$ the
\textbf{Schwartz embedding} of $\mathcal{L}_{loc}(\Omega)$ into
$\mathcal{D}^\prime(\Omega)$ defined by  $L_\Omega(f)=T_f$. Here
$T_f\in\mathcal{D}^\prime(\Omega)$ stands for the (regular) distribution with kernel $f$, i.e.
$(T_f|\tau)=\int_\Omega f(x)\tau(x)\, dx$ for all $\tau\in\mathcal{D}^\prime(\Omega)$
(Vladimirov~\cite{vVladimirov}). Also, $\mathcal{L}_{loc}(\Omega)$ denotes the space of the locally
integrable (Lebesgue) functions from $\Omega$ to $\mathbb{C}$. Recall that $L_\Omega$
preserves the addition and multiplication by complex numbers. The restriction of $L_\Omega$ on
$\mathcal{E}(\Omega)$ preserves also the partial differentiation (but not the multiplication). We shall write
$f\circledast\varphi$ and
$f\star\varphi$ instead of
$T_f\circledast\varphi$ and $T_f\star\varphi$, respectively. Thus  for every
$f\in\mathcal{L}_{loc}(\Omega)$, every
$\varphi\in\mathcal{D}_0$ and every $x\in\Omega$ we have
\begin{equation}\label{E: Lloc}
(f\circledast\varphi)(x)=\int_{||x-t||< R_\varphi} f(t)C_{\Omega,\varphi}(t)\varphi(x-t)\, dt.
\end{equation}
\end{enumerate}
\end{examples}

		In what follows we shall often write $K\Subset\Omega$ to indicate that $K$ is a compact subset of $\Omega$.

\begin{lemma}[Localization]\label{L: Localization} Let $\Omega$ be (as before) an open set of
$\mathbb{R}^d$ and
$T\in\mathcal{D}^\prime(\Omega)$ be a Schwartz distribution. Then for every compact set
$K\subset\Omega$ there exists
$n\in\mathbb{N}_0$ such that for every $x\in K$ and every $\varphi\in\mathcal{D}_n$ we have:

{\bf (a)}  $C_{\Omega,\varphi}(x)=1$. 

{\bf (b)} $(T\circledast\varphi)(x)=(T\star\varphi)(x)$.

{\bf (c)} Consequently, $(\forall K\Subset\Omega)(\forall\alpha\in\mathbb{N}_0^d)(\exists
n\in\mathbb{N}_0)(\forall x\in
K)(\forall\varphi\in\mathcal{D}_n)$ we have $\partial^\alpha (T\circledast\varphi)(x)=(\partial^\alpha
T\circledast\varphi)(x)=(T\circledast\partial^\alpha\varphi)(x)$.

\end{lemma}
\begin{proof} (a) Let $d(K,\partial\Omega)$ denote the Euclidean distance between $K$ and
$\partial\Omega$. It suffices to choose
$n\in\mathbb{N}$ such that
$3/n<d(K,\partial\Omega)$ and  $n>\sup_{x\in K}||x||+1$. It follows that  $3R_\varphi<d(K,\partial\Omega)$
for all
$\varphi\in\mathcal{D}_n$ because
$R_\varphi\leq 1/n$ holds by the definition of $\mathcal{D}_n$. Now (a) follows from the
property of the convolution (Vladimirov~\cite{vVladimirov}, Ch.I, \S4, 6.T).

(b) If $K\Subset\Omega$, then there exists $m\in\mathbb{N}$ such that $L=:\{t\in\Omega : d(t, K)\leq 1/m\}
\Subset\Omega$. Hence, by part (a), there exists $n\in\mathbb{N}$ (with $n\geq m$) such that
$C_{\Omega,\varphi}(x)\varphi(x-t)=\varphi(x-t)$ for all $x\in K$, all $t\in\Omega$ and all
$\varphi\in\mathcal{D}_n$.

	(c) follows directly from (b) bearing in mind that we have $\partial^\alpha (T\star\varphi)(x)=(\partial^\alpha
T\star\varphi)(x)=(T\star\partial^\alpha\varphi)(x)$.
\end{proof} 
\begin{lemma}[Schwartz Distributions]\label{L: Schwartz Distributions} Let $\Omega$ be an open set of 
$\mathbb{R}^d$ and $T\in\mathcal{D}^\prime(\Omega)$ be a Schwartz distribution. Then for every
compact set $K\subset\Omega$ and every multi-index
$\alpha\in\mathbb{N}_0^d$ there exist $m, n\in\mathbb{N}_0$ such that for every
$\varphi\in\mathcal{D}_n$ we have $
\sup_{x\in K}\left|\partial^\alpha(T\circledast\varphi)(x)\right|\leq (R_\varphi)^{-m}$.

\end{lemma}
\begin{proof}  Let $K$ and $\alpha$ be chosen arbitrarily. By Lemma~\ref{L: Localization}, there exists
$q\in\mathbb{N}$ such that $\partial^\alpha(T\circledast\varphi)(x)=(\partial^\alpha T\star\varphi)(x)$
for all $x\in K$ and all $\varphi\in\mathcal{D}_q$. Let
$\mathcal{O}$ be an open relatively compact subset of $\Omega$ containing $K$ and let
$k\in\mathbb{N}$ be greater than $1/d(K,\partial\mathcal{O})$. We observe that
$\varphi_x\in\mathcal{D}(\mathcal{O})$ for all $x\in K$ and all $\varphi\in\mathcal{D}_k$, where
$\varphi_x(t)=:\varphi(x-t)$. On the other hand, there exist
$M\in\mathbb{R}_+$ and $b\in\mathbb{N}_0$ such that
$\left|\left(\partial^\alpha T\left.\right|\tau\right)\right|\leq M\sum_{|\beta|\leq
b}\sup_{t\in\mathcal{O}}\left|\partial^\beta\tau(t)\right|$ for all $\tau\in \mathcal{D}(\mathcal{O})$ by
the continuity of $\partial^\alpha T$. Thus $\left|(\partial^\alpha
T\star\varphi)(x)\right|=\left|\left(\partial^\alpha T\left.\right|\varphi_x(t)\right)\right|\leq
M\sum_{|\beta|\leq b}\sup_{t\in\mathbb{R}^d}\left|\partial^\beta\varphi(t)\right|$ for all $x\in K$ and all
$\varphi\in\mathcal{D}_k$. With this in mind we choose $m=2(b+d)+1$ and 
$n\geq\max\{q, k, C, b\}$, where $C=M \sum_{|\beta|\leq b}1$. Now, for every $x\in K$ and every
$\varphi\in\mathcal{D}_n$ we have
\[
\left|\partial^\alpha(T\circledast\varphi)(x)\right|\leq
M \sum_{|\beta|\leq b}(R_\varphi)^{-2(|\beta|+d)}        
\leq C(R_\varphi)^{-2(b+d)}\leq (R_\varphi)^{-m},
\]
as required, where the last inequality holds because $R_\varphi\leq 1/n$ by the definition of
$\mathcal{D}_n$ (Definition~\ref{D: Directing Sets}) and  $1/n\leq 1/C$ 
by the choice of $n$.
\end{proof} 
\begin{lemma}[$\mathcal{C}^\infty$-Functions]\label{L: Cinfinity-Functions} Let $\Omega$ be an open set
of $\mathbb{R}^d$ and $f\in\mathcal{E}(\Omega)$ be a $\mathcal{C}^\infty$-function. Then for
every compact set $K\subset\Omega$, every multi-index $\alpha\in\mathbb{N}_0^d$ and every
$p\in\mathbb{N}$ there exists $n\in\mathbb{N}_0$ such that for every $\varphi\in\mathcal{D}_n$ we have
\[
\sup_{x\in K}\left|\partial^\alpha(f\circledast\varphi)(x)-\partial^\alpha
f(x)\right|\leq (R_\varphi)^{p}.
\]

\end{lemma}
\begin{proof}  Suppose that $p\in\mathbb{N}$, $K\Subset\Omega$ and
$\alpha\in\mathbb{N}_0^d$. By Lemma~\ref{L: Localization}, there exists $q\in\mathbb{N}_0$ such that
$\partial^\alpha (f\circledast\varphi)(x)=(\partial^\alpha f\star\varphi)(x)$ for all $x\in K$ and all
$\varphi\in\mathcal{D}_q$. As before, let
$\mathcal{O}$ be an open relatively compact subset of $\Omega$ containing $K$ and let
$k\in\mathbb{N}$ be greater than $1/d(K,\partial\mathcal{O})$. Let
$n\geq\max\left\{p, q, k,
\frac{2C}{(p+1)!}\right\}$, where
$C=:\sum_{|\beta|=p+1}\sup_{\xi\in \mathcal{O}}\left|(\partial^{\alpha+\beta}f)(\xi)\right|$. Let $x\in K$
and $\varphi\in\mathcal{D}_n$. By involving the definition of the sets $\mathcal{D}_n$, we calculate:
\begin{align}
&\left|\partial^\alpha (f\circledast\varphi)(x)-\partial^\alpha f(x)\right|=(\mbox{Lemma~\ref{L:
Localization} and } n\geq q)=\left|(\partial^\alpha f\star\varphi)(x)-\partial^\alpha f(x)\right|=\notag\\
&(\mbox{since } n\geq 1)=\left|\int_{||y||\leq R_\varphi}\left[\partial^\alpha
f(x-y)-\partial^\alpha f(x)\right]\varphi(y)\, dy\right|= \small{\left(\!\!{\text{Taylor
expansion}\atop\text{ for some }t\in[0,1]}\!\right)}=\notag
\end{align}
\begin{align}
&\left|\sum_{|\beta|=1}^p\frac{(-1)^{|\beta|}\partial^{\alpha+\beta}f(x)}{|\beta|!}\underbrace{\int_{||y||\leq
R_\varphi}y^\beta\varphi(y)\, dy}_{= 0 \mbox{ since } n\geq
p}+\frac{(-1)^{p+1}}{(p+1)!}\sum_{|\beta|=p+1}\int_{||y||\leq
R_\varphi}y^\beta\varphi(y)\partial^{\alpha+\beta}f(x-yt)\, dy\right|=\notag
\end{align}
\begin{align} 
&=\frac{R_\varphi^{p+1}}{(p+1)!}\left(C\int_{||y||\leq R_\varphi}|\varphi(y)|\,
dy\right)\leq \frac{R_\varphi^{p+1}}{(p+1)!}\,C\, (1+1/n)<\frac{R_\varphi^{p+1}}{(p+1)!}\,2C\leq
R_\varphi^p,\notag  
\end{align}
as required, where the last inequality follows from $R_\varphi\leq1/n\leq (p+1)!/2C$. 
\end{proof} 
\begin{lemma}[Pairing]\label{L: Pairing} Let $\Omega$ be an open set
of $\mathbb{R}^d$, $T\in\mathcal{D}^\prime(\Omega)$ be a Schwartz distribution and
$\tau\in\mathcal{D}(\Omega)$ be a test function. Then for every $p\in\mathbb{N}$ there exists
$n\in\mathbb{N}_0$ such that  for every $\varphi\in\mathcal{D}_n$ we have
\begin{equation}\label{E: Pairing}
\left|\left(T\circledast\varphi\left.\right|\tau\right)-
\left(T\left.\right|\tau\right)\right|\leq (R_\varphi)^{p}.
\end{equation}
\end{lemma}

\begin{proof} Let $p\in\mathbb{N}$ and let $\mathcal{O}$ be an
open relatively compact subset of $\Omega$ containing $\supp(\tau)$. There exist $M\in\mathbb{R}_+$ and
$a\in\mathbb{N}_0$ such that $|\left(T\left.\right|\psi \right)|\leq M\sum_{|\alpha|\leq
a}\sup_{x\in\mathcal{O}}|\partial^\alpha\psi(x)|$ for all $\psi\in\mathcal{D}(\mathcal{O})$ by the
continuity of $T$. Also, there exists $q\in\mathbb{N}_0$ such that
$\left|\partial^\alpha(\tau\circledast\varphi)(x)-\partial^\alpha
\tau(x)\right|\leq (R_\varphi)^{p+1}$
for all $x\in \overline{\mathcal{O}}$, all $|\alpha|\leq a$ and all $\varphi\in\mathcal{D}_q$ by Lemma~\ref{L:
Cinfinity-Functions}. We  observe as well that there exists $m\in\mathbb{N}_0$ such that
$\tau\circledast\varphi-\tau\in\mathcal{D}(\mathcal{O})$ whenever
$\varphi\in\mathcal{D}_m$. Let $\varphi\in\mathcal{D}_n$, where
$n\geq\max\{1, q, m, M\sum_{|\alpha|\leq a}1\}$. Since $\varphi(-x)=\varphi(x)$ for all
$x\in\mathbb{R}^d$, we have $\left|\left(T\circledast\varphi\left.\right|\tau\right)-
\left(T\left.\right|\tau\right)\right|=\left|\left(T\left.\right|\tau\circledast\varphi-\tau\right)\right|\leq
M\sum_{|\alpha|\leq a}(R_\varphi)^{p+1}=(R_\varphi)^{p}(R_\varphi)M(\sum_{|\alpha|\leq
a}1)\leq(R_\varphi)^{p}$ as required.
\end{proof} 
\section{Asymptotic Numbers and Asymptotic Functions}\label{S: Asymptotic Numbers and Asymptotic
Functions}

	\quad We define a field  $\widehat{\mathbb{C}^{\mathcal{D}_0}}$ of {\em asymptotic numbers} and 
the differential algebra of {\em asymptotic functions} $\widehat{\mathcal{E}(\Omega)^{\mathcal{D}_0}}$ over
the field $\widehat{\mathbb{C}^{\mathcal{D}_0}}$. No background in non-standard analysis 
is required of the reader: our framework is still the usual standard analysis. Both
$\widehat{\mathbb{C}^{\mathcal{D}_0}}$ and  $\widehat{\mathcal{E}(\Omega)^{\mathcal{D}_0}}$,
however, do have alternative non-standard representations, but we shall postpone the discussion of the
connection with non-standard analysis to Section~\ref{S: J.F. Colombeau's Non-Linear Theory of
Generalized Functions and Non-Standard Analysis}.

 	 The readers who are unfamiliar with the non-linear theory of generalized functions 
(Colombeau~\cite{jfCol84a}-\cite{jfCol92}) might treat this and the next  sections as an introduction
to a (modified and improved version) of  Colombeau theory. The readers who are familiar with Colombeau
theory will observe the strong similarity between the construction of
$\widehat{\mathbb{C}^{\mathcal{D}_0}}$ and the definition of the ring $\overline{\mathbb{C}}$  of 
Colombeau generalized numbers (Colombeau~\cite{jfCol84a}, pp. 136). The
definition of $\widehat{\mathcal{E}(\Omega)^{\mathcal{D}_0}}$ also resembles the definition of the special 
algebra $\mathcal{G}(\Omega)$ of Colombeau generalized functions (Colombeau~\cite{jfCol85}).
We believe, however, that our asymptotic numbers and asymptotic functions offer an important
improvement of Colombeau theory because $\widehat{\mathbb{C}^{\mathcal{D}_0}}$ is an algebraically
closed field (Theorem~\ref{T: Algebraic Properties}) in contrast to
$\overline{\mathbb{C}}$, which is a ring with zero divisors.
\begin{definition}[Asymptotic Numbers]\label{D: Asymptotic Numbers} Let $R_\varphi$ be the radius of
support of $\varphi$ (cf.(\ref{E: RadiusSupport})).
\begin{enumerate} 
\item\label{No: Moderate} We define the sets of the
\textbf{moderate} and \textbf{negligible} nets in $\mathbb{C}^{\mathcal{D}_0}$ by
\begin{align}
&\mathcal{M}(\mathbb{C}^{\mathcal{D}_0})=\left\{(A_\varphi)\in{\mathbb{C}^{\mathcal{D}_0}} : \;
(\exists m\in\mathbb{N})\left(|A_\varphi|\leq (R_\varphi)^{-m}\mbox{ a.e.}\right) \right\},
\label{E: Moderate}\\ 
&\mathcal{N}(\mathbb{C}^{\mathcal{D}_0})=\left\{(A_\varphi)\in{\mathbb{C}^{\mathcal{D}_0}}:\;
(\forall p\in\mathbb{N})\left(|A_\varphi|< (R_\varphi)^{p} \mbox{ a.e.}\right)\right\}, \label{E: Negligible}
\end{align}
respectively, where ``a.e'' stands for ``almost everywhere'' (Definition~\ref{D: Almost Everywhere}).  We
define the factor ring  $\widehat{\mathbb{C}^{\mathcal{D}_0}}=:
\mathcal{M}(\mathbb{C}^{\mathcal{D}_0})/\mathcal{N}(\mathbb{C}^{\mathcal{D}_0})$ and we
denote by
$\widehat{A_\varphi}\in\widehat{\mathbb{C}^{\mathcal{D}_0}}$ the equivalence class of
the net $(A_\varphi)\in\mathcal{M}(\mathbb{C}^{\mathcal{D}_0})$.

\item If $\mathcal{S}\subseteq\mathbb{C}^{\mathcal{D}_0}$, we let 
$\widehat{\mathcal{S}}=:\left\{\widehat{A_\varphi} :
(A_\varphi)\in\mathcal{S}\cap\mathcal{M}(\mathbb{C}^{\mathcal{D}_0})\right\}$. We call the
elements of  $\widehat{\mathbb{C}^{\mathcal{D}_0}}$ \textbf{complex asymptotic numbers}
 and the elements of $\widehat{\mathbb{R}^{\mathcal{D}_0}}$
\textbf{real asymptotic numbers}. We define an \textbf{order relation} on
$\widehat{\mathbb{R}^{\mathcal{D}_0}}$ as follows: Let
$\widehat{A_\varphi}\in\widehat{\mathbb{R}^{\mathcal{D}_0}}$ and
$\widehat{A_\varphi}\not= 0$. Then  
$\widehat{A_\varphi}>0$  if $A_\varphi>0$ a.e., that is 
$\{\varphi\in\mathcal{D}_0 : A_\varphi>0\}\in\mathcal{U}$.
\item We define the embeddings $\mathbb{C}\subset\widehat{\mathbb{C}^{\mathcal{D}_0}}$ and 
$\mathbb{R}\subset\widehat{\mathbb{R}^{\mathcal{D}_0}}$ by the constant nets, i.e. by
$A\to\widehat{A}$.
\end{enumerate}
\end{definition} 
\begin{theorem}[Algebraic Properties]\label{T: Algebraic Properties}
$\widehat{\mathbb{C}^{\mathcal{D}_0}}$ is an algebraically closed field, 
$\widehat{\mathbb{R}^{\mathcal{D}_0}}$ is a real closed field and we have the usual
connection
$\widehat{\mathbb{C}^{\mathcal{D}_0}}=\widehat{\mathbb{R}^{\mathcal{D}_0}}\oplus
i\,\widehat{\mathbb{R}^{\mathcal{D}_0}}$. 
\end{theorem}
\begin{proof}  It is clear that
$\widehat{\mathbb{C}^{\mathcal{D}_0}}$ is a ring and
$\widehat{\mathbb{C}^{\mathcal{D}_0}}=
\widehat{\mathbb{R}^{\mathcal{D}_0}}+i\,\widehat{\mathbb{R}^{\mathcal{D}_0}}$. To show that 
$\widehat{\mathbb{C}^{\mathcal{D}_0}}$ is a field, suppose that
$(A_\varphi)\in\mathcal{M}(\mathbb{C}^{\mathcal{D}_0})\setminus
\mathcal{N}(\mathbb{C}^{\mathcal{D}_0})$. Thus there exist $m, p\in\mathbb{N}$ such that 
$\Phi =: \{\varphi\in\mathcal{D}_0 : (R_\varphi)^p\leq |A_\varphi|\leq
(R_\varphi)^{-m}\}\in\mathcal{U}$. We define the net $(B_\varphi)\in\mathbb{C}^{\mathcal{D}_0}$ by 
$B_\varphi=1/A_\varphi$ if $\varphi\in\Phi$ and $B_\varphi=1$ if
$\varphi\in\mathcal{D}_0\setminus\Phi$. It is clear that $A_\varphi\, B_\varphi=1$ a.e. thus
$\widehat{A_\varphi}\, \widehat{B_\varphi}=1$ as required. To show that
$\widehat{\mathbb{C}^{\mathcal{D}_0}}$ is an algebraically closed field, let
$P(x)=x^p+a_{n-1}x^{p-1}+\cdots +a_0$ be a polynomial with coefficients in
$\widehat{\mathbb{C}^{\mathcal{D}_0}}$ and degree $p\geq 1$. Since
$\widehat{\mathbb{C}^{\mathcal{D}_0}}$ is a field, we have assumed without loss of generality that the leading
coefficient is $1$. We have 
$a_k=\widehat{A_{\varphi, k}}$, for some moderate nets $(A_{\varphi, k})$. Denote
$P_\varphi(x) =: x^p+A_{\varphi, p-1}x^{p-1}+\cdots +A_{\varphi, 0}$ and observe that for every
$\varphi\in\mathcal{D}_0$ there exists a complex number $X_\varphi\in\mathbb{C}$ such that
$P_\varphi(X_\varphi)=0$ since $\mathbb{C}$ is an algebraically closed field. Thus there exists a net
$(X_\varphi)\in\mathbb{C}^{\mathcal{D}_0}$ such that
$P(X_\varphi)=0$ for all $\varphi\in\mathcal{D}_0$. Also the estimation $|X_\varphi|\leq 
1+|A_{\varphi, p-1}|+\cdots +|A_{\varphi, 0}|$ implies that the net $(X_\varphi)$ is a moderate net. The
asymptotic number $\widehat{X_\varphi}\in\widehat{\mathbb{C}^{\mathcal{D}_0}}$ is the zero of the
polynomial $P$ we are looking for because
$P(\widehat{X_\varphi})=\widehat{X_\varphi}^p+a_{p-1}\widehat{X_\varphi}^{p-1}+\cdots
+a_0=\widehat{X_\varphi}^p+\widehat{A_{\varphi, p-1}}\widehat{X_\varphi}^{p-1}+\cdots
+\widehat{A_{\varphi, 0}}=\widehat{P_\varphi (X_\varphi)}=\widehat{0}=0$ as required. The fact that
$\widehat{\mathbb{R}^{\mathcal{D}_0}}$ is a real closed field follows directly from the fact that
$\widehat{\mathbb{C}^{\mathcal{D}_0}}$ is an algebraically closed field and the connection
$\widehat{\mathbb{C}^{\mathcal{D}_0}}=
\widehat{\mathbb{R}^{\mathcal{D}_0}}+i\,\widehat{\mathbb{R}^{\mathcal{D}_0}}$
(Van Der Waerden~\cite{VanDerWaerden}, Chapter 11). 
\end{proof} 
\begin{corollary}[Total Order]\label{C: Total Order} $\widehat{\mathbb{R}^{\mathcal{D}_0}}$ is a
\textbf{totally ordered field} and we have the following characterization of the order relation: if 
$a\in\widehat{\mathbb{R}^{\mathcal{D}_0}}$ then $a\geq0$ \ifff $a=b^2$ for some
$b\in\widehat{\mathbb{R}^{\mathcal{D}_0}}$. Consequently, the mapping $|\cdot|:
\widehat{\mathbb{C}^{\mathcal{D}_0}}\to\widehat{\mathbb{R}^{\mathcal{D}_0}}$, defined by the formula
$|a+ib|=\sqrt{a^2+b^2}$, is an \textbf{absolute value} on $\widehat{\mathbb{C}^{\mathcal{D}_0}}$
(Ribenboim~\cite{pRib}, pp.3-6).
\end{corollary}
\begin{proof} The algebraic operations in any
real closed field uniquely determine a total order (Van Der Waerden~\cite{VanDerWaerden}, Chapter 11). Thus
the characterization of the order relation in $\widehat{\mathbb{R}^{\mathcal{D}_0}}$ follows directly from
the fact that $\widehat{\mathbb{R}^{\mathcal{D}_0}}$ is a real closed field.  The
existence of the root 
$\sqrt{x}$ for any non-negative $x$ in $\widehat{\mathbb{R}^{\mathcal{D}_0}}$ also follows from the fact
that $\widehat{\mathbb{R}^{\mathcal{D}_0}}$ is a real closed field.
\end{proof} 
\begin{definition}[Infinitesimals, Finite and Infinitely Large]
\label{D: Infinitesimals, Finite and Infinitely Large}  An asymptotic number
$z\in\widehat{\mathbb{C}^{\mathcal{D}_0}}$ is called
\textbf{infinitesimal}, in symbol $z\approx 0$, if $|z|<1/n$ for all $n\in\mathbb{N}$.  Similarly, $z$ is called
\textbf{finite} if $|z|<n$ for some $n\in\mathbb{N}$. And $z$ is \textbf{infinitely large} if $n<|z|$ for all
$n\in\mathbb{N}$. We denote by $\mathcal{I}(\widehat{\mathbb{C}^{\mathcal{D}_0}}),
\mathcal{F}(\widehat{\mathbb{C}^{\mathcal{D}_0}})$ and
$\mathcal{L}(\widehat{\mathbb{C}^{\mathcal{D}_0}})$ the sets of the infinitesimal, finite and infinitely large
numbers in
$\widehat{\mathbb{C}^{\mathcal{D}_0}}$, respectively. We define the \textbf{standard part mapping} 
$\widehat{\st}: \mathcal{F}(\widehat{\mathbb{C}^{\mathcal{D}_0}})\to
\mathbb{C}$ by the formula $\widehat{\st}(z)\approx z$.
\end{definition}

	 The next result shows that both $\widehat{\mathbb{R}^{\mathcal{D}_0}}$ and
$\widehat{\mathbb{C}^{\mathcal{D}_0}}$ are non-archimedean fields in the sense that they contain
non-zero infinitesimals.
\begin{lemma}[Canonical Infinitesimal in $\widehat{\mathbb{R}^{\mathcal{D}_0}}$] \label{D: Canonical
Infinitesimal in widehatRD0} Let $R_\varphi$ be the radius of
support of $\varphi$ (cf.(\ref{E: RadiusSupport})).  Then the asymptotic number
$\widehat{\rho}=:\widehat{R_\varphi}$ is a positive infinitesimal in
$\widehat{\mathbb{R}^{\mathcal{D}_0}}$.  We call  $\widehat{\rho}$ the
\textbf{canonical infinitesimal} in $\widehat{\mathbb{R}^{\mathcal{D}_0}}$ (the choice of the notation 
$\widehat{\rho}$ will be justified in Section~\ref{S: J.F. Colombeau's Non-Linear Theory of Generalized
Functions and Non-Standard Analysis}).
\end{lemma}
\begin{proof}  We have $0\leq\widehat{\rho}<1/n$ for all
$n\in\mathbb{N}$ because $R_\varphi\in\mathbb{R}_+\, \&\,
R_\varphi<1/n$ a.e.  (cf. Example~\ref{Ex: Radius of Support}). Also, $\widehat{\rho}\not=0$ because
$(R_\varphi)\notin\mathcal{N}(\mathbb{C}^{\mathcal{D}_0})$.
\end{proof} 

\begin{definition}[Topology, Valuation, Ultra-Norm,  Ultra-Metric]\label{D: Topology, Valuation, Ultra-Norm,
Ultra-Metric}  We supply $\widehat{\mathbb{C}^{\mathcal{D}_0}}$ with the \textbf{order topology},
i.e. the product topology inherited from the  order topology on $\widehat{\mathbb{R}^{\mathcal{D}_0}}$. 
We define a
\textbf{valuation}
$v:\widehat{\mathbb{C}^{\mathcal{D}_0}}\to\mathbb{R}\cup\{\infty\}$  on
$\widehat{\mathbb{C}^{\mathcal{D}_0}}$ by 
$v(z)=\sup\{q\in\mathbb{Q}\mid z/\widehat{\rho}\,^q\approx 0\}$ if $z\not= 0$ and $v(0)=\infty$.
We define  the
\textbf{ultra-norm}
$|\cdot|_v: \widehat{\mathbb{C}^{\mathcal{D}_0}}\to\mathbb{R}$ by the formula $|z|_v=e^{-v(z)}$ (under
the convention that $e^{-\infty}=0$). The formula
$d(a, b)=|a-b|_v$ defines an \textbf{ultra-metric} on $\widehat{\mathbb{C}^{\mathcal{D}_0}}$.
\end{definition}
\begin{theorem}[Ultra-Properties]\label{T: Ultra-Properties} Let $a, b, c\in
{\widehat{\mathbb{C}^{\mathcal{D}_0}}}$. Then
\begin{description}
\item[(i)]  {\bf (a)} $v(a)=\infty$ \ifff $a=0$; 
{\bf (b)} $v(ab)=v(a)+v(b)$. {\bf (c)} $v(a+b)\geq\min\{v(a), v(b)\}$;
{\bf (d)} $|a|<|b|$ implies $v(a)\geq v(b)$.
\item[(ii)] {\bf (a)} $|0|_v=0,  |\pm1|_v=1,\; \text{and}\; 
|a|_v>0\; \text{whenever }  a\not=0$ ; {\bf (b)}  $|ab|_v=|a|_v\, |b|_v$; 
{\bf (c)} $|a+b|_v\leq\max\{|a|_v, |b|_v\}$ (\textbf{ultra-norm inequality}); {\bf (d)}
$|a|<|b|\;\text{implies}\;  |a|_v\leq |b|_v$.
\item[(iii)]  $d(a,b)\leq\max\{
d(a,c), d(c, b)\}$ (\textbf{ultra-metric inequality}). Consequently, 
$(\widehat{\mathbb{C}^{\mathcal{D}_0}}, d)$ and
$(\widehat{\mathbb{R}^{\mathcal{D}_0}}, d)$ are {\bf ultra-metric spaces}.
\end{description}
\end{theorem}
\begin{proof}  The properties (i)-(iii) follow easily from the definition of $v$ and we leave the verification to
the reader.
\end{proof} 
\begin{remark}[Colombeau Theory]\label{R: Colombeau Theory}  The counterpart $\bar{v}$ of $v$ in
Colombeau theory is only a pseudo-valuation, not a valuation, in the sense that $\bar{v}$ satisfies the
property
$\bar{v}(ab)\geq \bar{v}(a)+\bar{v}(b)$, not $v(ab)=v(a)+v(b)$. Consequently, the the counterpart
$|\cdot|_{\bar{v}}$ of
$|\cdot|_v$ in Colombeau theory is pseudo-ultra-metric, not a ultra-metric, in the sense that it satisfies the
property
$|ab|_{\bar{v}}\leq |a|_{\bar{v}}\, |b|_{\bar{v}}$, not $|ab|_v=|a|_v\, |b|_v$. For the concept of {\em classical
valuation} we refer the reader to (Ribenboim~\cite{pRib}).

\end{remark}
\begin{definition}[Asymptotic Functions]\label{D: Asymptotic Functions}  Let $\Omega$ be an open set
of
$\mathbb{R}^d$ and $R_\varphi$
be the radius of support of $\varphi$ (cf.(\ref{E: RadiusSupport})). 
\begin{enumerate} 
\item We define the sets of the
\textbf{moderate nets} $\mathcal{M}(\mathcal{E}(\Omega)^{\mathcal{D}_0})$ and \textbf{negligible nets}
$\mathcal{N}(\mathcal{E}(\Omega)^{\mathcal{D}_0})$ of
$\mathcal{E}(\Omega)^{\mathcal{D}_0}$ by:
$(f_\varphi)\in\mathcal{M}(\mathcal{E}(\Omega)^{\mathcal{D}_0})$ if (by definition)
\begin{equation}\label{E: rho-moderate}\notag
(\forall K\Subset\Omega)(\forall\alpha\in\mathbb{N}^d)(\exists
m\in\mathbb{N}_0)(\sup_{x\in K}|\partial^\alpha f_\varphi(x)|\leq (R_\varphi)^{-m}\mbox{ a.e.}),
\end{equation}
and, similarly, $(f_\varphi)\in\mathcal{N}(\mathcal{E}(\Omega)^{\mathcal{D}_0})$ if (by definition)
\begin{equation}\label{E: rho-moderate}\notag
(\forall K\Subset\Omega)(\forall\alpha\in\mathbb{N}^d)(\forall
p\in\mathbb{N})(\sup_{x\in K}|\partial^\alpha f_\varphi(x)|\leq (R_\varphi)^{p}\mbox{ a.e.}),
\end{equation}
respectively. Here  $\partial^\alpha f_\varphi(x)$ stands for the
$\alpha$-partial derivative  of
$f_\varphi(x)$ with respect to $x$ and ``a.e'' stands
(as before) for ``almost everywhere'' (Definition~\ref{D: Almost Everywhere}).  We define the factor ring 
$\widehat{\mathcal{E}(\Omega)^{\mathcal{D}_0}}=:
\mathcal{M}(\mathcal{E}(\Omega)^{\mathcal{D}_0})/\mathcal{N}\left(\mathcal{E}(\Omega)^
{\mathcal{D}_0}\right)$
and we denote by
$\widehat{f_\varphi}\in\widehat{\mathcal{E}(\Omega)^{\mathcal{D}_0}}$ the equivalence
class of the net $(f_\varphi)\in\mathcal{M}(\mathcal{E}(\Omega)^{\mathcal{D}_0})$. We call the
elements of  $\widehat{\mathcal{E}(\Omega)^{\mathcal{D}_0}}$ \textbf{asymptotic functions} on $\Omega$.
More generally, if $\mathcal{S}\subseteq\mathcal{E}(\Omega)^{\mathcal{D}_0}$, we let 
$\widehat{\mathcal{S}}=:\left\{\widehat{f_\varphi} :
(f_\varphi)\in\mathcal{S}\cap\mathcal{M}(\mathcal{E}(\Omega)^{\mathcal{D}_0})\right\}$. 
\item We supply  $\widehat{\mathcal{E}(\Omega)^{\mathcal{D}_0}}$ with the \textbf{ring
operations and partial differentiation} of any order inherited from
$\mathcal{E}(\Omega)$. Also, for every asymptotic number
$\widehat{A_\varphi}\in\widehat{\mathbb{C}^{\mathcal{D}_0}}$ and every asymptotic function
$\widehat{f_\varphi}\in\widehat{\mathcal{E}(\Omega)^{\mathcal{D}_0}}$ we
define the \textbf{product} $\widehat{A_\varphi}\,
\widehat{f_\varphi}\in\widehat{\mathcal{E}(\Omega)^{\mathcal{D}_0}}$ by 
$\widehat{A_\varphi}\, \widehat{f_\varphi}=\widehat{A_\varphi\, f_\varphi}$.
\item We define the \textbf{pairing} between $\widehat{\mathcal{E}(\Omega)^{\mathcal{D}_0}}$ and
$\mathcal{D}(\Omega)$ by the formula $(\widehat{f_\varphi}|\tau)=$
$\widehat{(f_\varphi|\tau)}$, where 
$(f_\varphi|\tau)=: \int_\Omega f_\varphi(x)\tau(x)\, dx$. 
\item\label{No: Weakly Equal} We say that an asymptotic function
$\widehat{f_\varphi}\in\widehat{\mathcal{E}(\Omega)^{\mathcal{D}_0}}$  is \textbf{weakly equal to zero} in
$\widehat{\mathcal{E}(\Omega)^{\mathcal{D}_0}}$, in symbol
$\widehat{f_\varphi}\cong 0$, if $(\widehat{f_\varphi}|\tau)=0$ for all
$\tau\in\mathcal{D}(\Omega)$.  We say
that $\widehat{f_\varphi}, \widehat{g_\varphi}\in\widehat{\mathcal{E}(\Omega)^{\mathcal{D}_0}}$ are
\textbf{weakly equal}, in symbol
$\widehat{f_\varphi}\cong\widehat{g_\varphi}$, if
$(\widehat{f_\varphi}|\tau)=(\widehat{g_\varphi}|\tau)$ in
$\widehat{\mathbb{C}^{\mathcal{D}_0}}$ for all $\tau\in\mathcal{D}(\Omega)$. 

\item\label{No: Weakly Infinitesimal} We say that an asymptotic function
$\widehat{f_\varphi}\in\widehat{\mathcal{E}(\Omega)^{\mathcal{D}_0}}$  is \textbf{weakly infinitesimal} (or,
\textbf{associated to zero}), in symbol
$\widehat{f_\varphi}\approx 0$, if $(\widehat{f_\varphi}|\tau)\approx 0$ for all
$\tau\in\mathcal{D}(\Omega)$, where the latter $\approx$ is the infinitesimal relation on
$\widehat{\mathbb{C}^{\mathcal{D}_0}}$ (Definition~\ref{D: Canonical Infinitesimal in widehatRD0}). 
We
say that $\widehat{f_\varphi}, \widehat{g_\varphi}\in\widehat{\mathcal{E}(\Omega)^{\mathcal{D}_0}}$ are
\textbf{weakly infinitesimal} (or, \textbf{associated}), in symbol
$\widehat{f_\varphi}\approx\widehat{g_\varphi}$, if
$(\widehat{f_\varphi}|\tau)\approx (\widehat{g_\varphi}|\tau)$ for all $\tau\in\mathcal{D}(\Omega)$, where
in the latter formula $\approx$ stands for the infinitesimal relation in
$\widehat{\mathbb{C}^{\mathcal{D}_0}}$. 

\item Let $\widehat{f_\varphi}\in\widehat{\mathcal{E}(\Omega)^{\mathcal{D}_0}}$ and
let $\mathcal{O}$ be an open subset of $\Omega$. We define the \textbf{restriction}
$\widehat{f_\varphi}\rest\mathcal{O}\in\widehat{\mathcal{E}(\mathcal{O})^{\mathcal{D}_0}}$ of
$\widehat{f_\varphi}$ to $\mathcal{O}$ by
$\widehat{f_\varphi}\rest\mathcal{O}=\widehat{f_\varphi\rest\mathcal{O}}$, where
$f_\varphi\rest\mathcal{O}$ is the usual restriction of $f_\varphi$ to $\mathcal{O}$. The
\textbf{support}
$\supp(\widehat{f_\varphi})$ of $\widehat{f_\varphi}$ is the complement to $\Omega$ of the largest open
subset $G$ of $\Omega$ such that $\widehat{f_\varphi}\rest G=0$ in
$\widehat{\mathcal{E}(G)^{\mathcal{D}_0}}$. 
\item Let $\Omega, \Omega^\prime\in\mathcal{T}^d$ and
$\psi\in\Diff(\Omega^\prime, \Omega)$ be a diffeomorphism. For every
$\widehat{f_\varphi}\in\widehat{\mathcal{E}(\Omega)^{\mathcal{D}_0}}$ we define the \textbf{composition}
(or, change of variables)
$\widehat{f_\varphi}\circ\psi\in\widehat{\mathcal{E}(\Omega^\prime)^{\mathcal{D}_0}}$ by the formula
$\widehat{f_\varphi}\circ\psi=\widehat{f_\varphi\circ\psi}$, where $f_\varphi\circ\psi$ stands  for the 
usual composition  between $f_\varphi$ and $\psi$.
\end{enumerate}
\end{definition} 

	 It is clear that $\mathcal{M}(\mathcal{E}(\Omega)^{\mathcal{D}_0})$ is a differential ring and
$\mathcal{N}(\mathcal{E}(\Omega)^{\mathcal{D}_0})$ is a differential ideal in
$\mathcal{M}(\mathcal{E}(\Omega)^{\mathcal{D}_0})$. Thus
$\widehat{\mathcal{E}(\Omega)^{\mathcal{D}_0}}$ is a \textbf{differential ring}. We leave to the reader to
verify that the product $\widehat{A_\varphi}\, \widehat{f_\varphi}$ is correctly defined. Thus we have the
following result:
\begin{theorem}[Differential Algebra]\label{T: Differential Algebra}
$\widehat{\mathcal{E}(\Omega)^{\mathcal{D}_0}}$ is a
\textbf{differential algebra over the field} $\widehat{\mathbb{C}^{\mathcal{D}_0}}$. 
\end{theorem}
\section{A Solution to the Problem of Multiplication of Schwartz Distributions}
\label{S: A Solution to the Problem of Multiplication of Schwartz Distributions}
\quad In this section we construct a {\em canonical} {\em embedding} $E_\Omega$ of the space
$\mathcal{D}^\prime(\Omega)$  of Schwartz distributions into the algebra of asymptotic functions
$\widehat{\mathcal{E}(\Omega)^{\mathcal{D}_0}}$. Thus $\widehat{\mathcal{E}(\Omega)^{\mathcal{D}_0}}$
becomes a full algebra of generalized functions of Colombeau type (see the Introduction).

	 The algebra of asymptotic functions
$\widehat{\mathcal{E}(\Omega)^{\mathcal{D}_0}}$ supplied with the embedding $E_\Omega$ offers a
solution to the problem of the multiplication of Schwartz distributions similar to but different from
Colombeau's solution (Colombeau~\cite{jfCol84a}). 

\begin{definition}[Embeddings]\label{D: Embeddings}  Let $\Omega$ be an open set of 
$\mathbb{R}^d$. 
\begin{enumerate} 
\item The \textbf{standard embedding}
$\sigma_\Omega: \mathcal{E}(\Omega)\to\widehat{\mathcal{E}(\Omega)^{\mathcal{D}_0}}$ is defined by
the constant nets, i.e. by the formula $\sigma_\Omega(f)=\widehat{f}$.
\item The \textbf{distributional embedding}
$E_\Omega: \mathcal{D}^\prime(\Omega)\to\widehat{\mathcal{E}(\Omega)^{\mathcal{D}_0}}$ is
defined by the formula $E_\Omega(T)=\widehat{T\circledast\varphi}$, where  
$T\circledast\varphi$ is the $\varphi$-regularization of $T\in \mathcal{D}^\prime(\Omega)$
(\# \ref{No: phi-regularization} in Examples~\ref{Exs: Nets and Distributions}).
\item The \textbf{classical function embedding} $E_\Omega\circ L_\Omega:
\mathcal{L}_{loc}(\Omega)\to\widehat{\mathcal{E}(\Omega)^{\mathcal{D}_0}}$ is defined by the formula
$(E_\Omega\circ L_\Omega)(f)=\widehat{f\circledast\varphi}$, where  $f\circledast\varphi$ is the
$\varphi$-regularization of $f\in \mathcal{L}_{loc}(\Omega)$ (\# \ref{No: Loc} in
Examples~\ref{Exs: Nets and Distributions}).
\end{enumerate}
\end{definition}

\begin{lemma}[Correctness]\label{L: Correctness} The constant nets are moderate in the sense that
$f\in\mathcal{E}(\Omega)$ implies $(f)\in\mathcal{M}(\mathcal{E}(\Omega)^{\mathcal{D}_0})$ 
(Section~\ref{S: Asymptotic Numbers and Asymptotic Functions}). Similarly
the $\varphi$-regularization of a Schwartz distribution (\# \ref{No: phi-regularization} in Examples~\ref{Exs:
Nets and Distributions}) is also a moderate net, i.e. $T\in\mathcal{D}^\prime(\Omega)$ implies
$(T\circledast\varphi)\in\mathcal{M}(\mathcal{E}(\Omega)^{\mathcal{D}_0})$.
\end{lemma}
\begin{proof}  It is clear that the constant nets are moderate. To show the moderateness of
$(T\circledast\varphi)$, suppose that $K\Subset\Omega$ and $\alpha\in\mathbb{N}_0$. By Lemma~\ref{L:
Schwartz Distributions} there exist $m, n\in\mathbb{N}_0$ such that
$\mathcal{D}_n\subseteq\{\varphi\in\mathcal{D}_0 : 
(\forall x\in K)\left|\partial^\alpha(T\circledast\varphi)(x)\right|\leq (R_\varphi)^{-m}\}$ implying
$\{\varphi\in\mathcal{D}_0 : 
\sup_{x\in K}\left|\partial^\alpha(T\circledast\varphi)(x)\right|\leq (R_\varphi)^{-m}\}\in\mathcal{U}$, as
required.  
\end{proof} 
	
	Notice that the embedding $E_\Omega$ is {\em canonical} in the sense that it is uniquely defined in terms
already used in the definition of the family
$\left\{\widehat{\mathcal{E}(\Omega)^{\mathcal{D}_0}}\right\}_{\Omega\in\mathcal{T}^d}$
(Definition~\ref{D: Asymptotic Functions}).
\begin{theorem}[Properties of Embedding]\label{T: Properties of Embedding} Let $\Omega$ be an open set
of 
$\mathbb{R}^d$. Then:
\begin{description}
\item[(i)] We have $(E_\Omega\circ L_\Omega)(f)=\sigma_\Omega(f)$ for all
$f\in\mathcal{E}(\Omega)$. 
This can be summarized in the following commutative diagram:
\[ 
\begin{diagram} 
\node{\mathcal{E}(\Omega)} \arrow{e,t}{L_\Omega} \arrow{s,l}{\sigma_\Omega}
\node{\mathcal{D^\prime}(\Omega)} \arrow{sw,r}{E_{\Omega}} \\ 
\node{\widehat{\mathcal{E}(\Omega)^{\mathcal{D}_0}}}
\end{diagram} 
\] 

Consequently,
$\mathcal{E}(\Omega)$ and $(E_\Omega\circ L_\Omega)[\mathcal{E}(\Omega)]$ are isomorphic
differential algebras over $\mathbb{C}$. Also, $E_\Omega\circ
L_\Omega=\sigma_\Omega$ preserves the pairing between $\mathcal{E}(\Omega)$ and
$\mathcal{D}(\Omega)$ in the sense that 
\[
\int_\Omega
f(x)\tau(x)\, dx=\left(\sigma_\Omega(f)\left.\right|\tau\right)=\left((E_\Omega\circ
L_\Omega)(f)\left.\right|\tau\right),
\] 
for all $f\in\mathcal{E}(\Omega)$ and all
$\tau\in\mathcal{D}(\Omega)$. Consequently, $E_\Omega\circ
L_\Omega=\sigma_\Omega$ is injective.
\item[(ii)] $E_\Omega$ is $\mathbb{C}$-linear and it preserves the partial differentiation of any order in
$\mathcal{D}^\prime(\Omega)$. Also, 
$E_\Omega$ preserves the pairing between
$\mathcal{D}^\prime(\Omega)$ and $\mathcal{D}(\Omega)$ in the sense that
$\left(T\left.\right|\tau\right)=\left(E_\Omega(T)\left.\right|\tau\right)$ for all
$T\in\mathcal{D}^\prime(\Omega)$ and all $\tau\in\mathcal{D}(\Omega)$.
Consequently, $E_\Omega$ is injective. 

\item[(iii)] $E_\Omega\circ L_\Omega$ is $\mathbb{C}$-linear. Also, $E_\Omega\circ L_\Omega$
 preserves the pairing between $\mathcal{L}_{loc}(\Omega)$ and $\mathcal{D}(\Omega)$
in the sense that 
\[
\int_\Omega f(x)\tau(x)\, dx=\left((E_\Omega\circ L_\Omega)(f)\left.\right|\tau\right),
\] 
for all $f\in\mathcal{L}_{loc}(\Omega)$ and  all  $\tau\in\mathcal{D}(\Omega)$. Consequently,
$E_\Omega\circ L_\Omega$ is injective.
\item[(iv)] Each of the above embeddings: $\sigma_\Omega, E_\Omega$ and $E_\Omega\circ
L_\Omega$,  is \textbf{sheaf preserving} in the sense that it preserves the restriction to an open subset. 
\end{description}
	We \textbf{summarize} all of the above
in $\mathcal{E}(\Omega)\subset\mathcal{L}_{loc}(\Omega)\subset\mathcal{D}^\prime(\Omega)\subset
\widehat{\mathcal{E}(\Omega)^{\mathcal{D}_0}}$, where: (a) $\mathcal{E}(\Omega)$ is a \textbf{differential
subalgebra} of $\widehat{\mathcal{E}(\Omega)^{\mathcal{D}_0}}$ over $\mathbb{C}$; (b)
$\mathcal{L}_{loc}(\Omega)$  is a \textbf{vector subspace} of
$\widehat{\mathcal{E}(\Omega)^{\mathcal{D}_0}}$ over
$\mathbb{C}$ and (c)
$\mathcal{D}^\prime(\Omega)$  is a \textbf{differential
vector subspace} of $\widehat{\mathcal{E}(\Omega)^{\mathcal{D}_0}}$ over $\mathbb{C}$. We shall often
write simply $T$ instead of the more precise $E_\Omega(T)$ for a Schwartz distribution in the
framework of $\widehat{\mathcal{E}(\Omega)^{\mathcal{D}_0}}$. 
\end{theorem}
\begin{proof}   (i) Suppose that $K\Subset\Omega$, $\alpha\in\mathbb{N}^d_0$ and $p\in\mathbb{N}$
(are chosen arbitrarily). By Lemma~\ref{L: Cinfinity-Functions} there exist $n\in\mathbb{N}_0$ such that
\[
\mathcal{D}_n\subseteq\left\{\varphi\in\mathcal{D}_0 : 
\sup_{x\in K}\left|\partial^\alpha(f\circledast\varphi)(x)-\partial^\alpha
f(x)\right|\leq (R_\varphi)^{p}\right\}.
\]
Thus $\{\varphi\in\mathcal{D}_0 :\sup_{x\in K}
\left| \partial^\alpha(f\circledast\varphi)(x)-\partial^\alpha f(x) \right|\leq
(R_\varphi)^p\}\in\mathcal{U}$. The latter means
that the net
$(f\circledast\varphi-f)$ is negligible (Definition~\ref{D: Asymptotic Functions}) thus $(E_\Omega\circ
L_\Omega)(f)=\widehat{f\circledast\varphi}=\widehat{f}=\sigma_\Omega(f)$ as required. Consequently, we
have $(E_\Omega\circ L_\Omega)[\mathcal{E}(\Omega)]=\sigma_\Omega[\mathcal{E}(\Omega)]$.
Thus $\mathcal{E}(\Omega)$ and $(E_\Omega\circ L_\Omega)[\mathcal{E}(\Omega)]$ are 
isomorphic differential algebras because $\mathcal{E}(\Omega)$ and
$\sigma_\Omega[\mathcal{E}(\Omega)]$ are (obviously) isomorphic differential algebras. Also,
$E_\Omega\circ L_\Omega$ preserves the pairing because
$\sigma_\Omega$ preserves (obviously) the pairing.

	(ii)  $\Sigma_\Omega$ is $\mathbb{C}$-linear because the mapping 
$T\to T\circledast\varphi$ is $\mathbb{C}$-linear. To show the preservation of partial differentiation we have
to show that for every multi-index $\beta\in\mathbb{N}^d_0$ the net
$\left(\partial^\beta T\circledast\varphi-\partial^\beta(T\circledast\varphi)\right)$ is negligible
(Definition~\ref{D: Asymptotic Functions}). This follows easily from Lemma~\ref{L: Localization} similarly to
(i) above. To show that
$E_\Omega$ preserves the pairing, we have to  show that for any test function $\tau$ 
the net $A_\varphi=:\left(T\circledast\varphi\left.\right|\tau\right)-\left(T\left.\right|\tau\right)$ is
negligible (Definition~\ref{D: Asymptotic Numbers}). The latter
 follows easily from
Lemma~\ref{L: Pairing}.

		(iii) $(E_\Omega\circ L_\Omega)$ is $\mathbb{C}$-linear because the mapping $f\to f\circledast\varphi$ is
$\mathbb{C}$-linear. The preserving of pairing follows from (ii) in the particular case $T=T_f$.

	(iv) The preserving of the restriction on an open subset follows easily
from the definition and we leave the  details to the reader.
\end{proof} 

 We should mention that if $f\in\mathcal{E}(\Omega)$ and
$T\in\mathcal{D}^\prime(\Omega)$, then $E_\Omega(f)E_\Omega(T)=E_\Omega(fT)$ is false in general. That
means that the multiplication in the algebra in
$\widehat{\mathcal{E}(\Omega)^{\mathcal{D}_0}}$ does not reproduce the Schwartz multiplication in 
$\mathcal{D}^\prime(\Omega)$ (multiplication by duality). Similarly, let $\mathcal{C}(\Omega)$ denote the
class of continuos functions from $\Omega$ to $\mathbb{C}$. If $g, h\in\mathcal{C}(\Omega)$, then
$E_\Omega(g)E_\Omega(h)=E_\Omega(gh)$ is also false in general. That means that the multiplication in the
algebra in
$\widehat{\mathcal{E}(\Omega)^{\mathcal{D}_0}}$ does not reproduce the usual multiplication in
$\mathcal{C}(\Omega)$. Of course, all these are inevitable in view of the Schwartz impossibility results
(Schwartz~\cite{lSchwartz54}). For a discussion we refer to (Colombeau~\cite{jfCol92}, p. 8).
Instead, we  have a  somewhat weaker result.
\begin{theorem}[Weak Preservation]\label{T: Weak Preservation} Let $T\in\mathcal{D}^\prime(\Omega),
f\in\mathcal{E}(\Omega)$ and
$g, h\in\mathcal{C}(\Omega)$. Then:
\begin{description}
\item[(i)] $E_\Omega(f)E_\Omega(T)\cong E_\Omega(fT)$ (Definition~\ref{D: Asymptotic Functions},
\#\ref{No: Weakly Equal}), i.e.
$\left(E_\Omega(f)E_\Omega(T)\left.\right|\tau\right)=\left(E_\Omega(fT)\left.\right|\tau\right)$ for all 
$\tau\in\mathcal{D}(\Omega)$.
\item[(ii)] $E_\Omega(g)E_\Omega(h)\approx E_\Omega(gh)$ (Definition~\ref{D: Asymptotic Functions},
\#\ref{No: Weakly Infinitesimal}), i.e.
$\left(E_\Omega(g)E_\Omega(h)\left.\right|\tau\right)\approx\left(E_\Omega(gh)\left.\right|\tau\right)$ for
all  $\tau\in\mathcal{D}(\Omega)$, where $\approx$ in the latter formula stands for the infinitesimal relation
in the field $\widehat{\mathbb{C}^{\mathcal{D}_0}}$.
\end{description}
\end{theorem}
\begin{proof} (i) We denote $f_{\varphi,\tau}:=\left(f (T\conv\varphi)\left.\right|\tau\right)=
\left(T\conv\varphi\left.\right|f\tau\right)$ and calculate\newline
$\left(E_\Omega(f)E_\Omega(T)\left.\right|\tau\right)=
\left(\widehat{f}\;\widehat{T\conv\varphi}\left.\right|\tau\right)=
\left(\widehat{f (T\conv\varphi)}\left.\right|\tau\right)=
\widehat{f_{\varphi,\tau}}=\left(\widehat{T\conv\varphi} \left.\right|f\tau\right)=
\left(T \left.\right|f\tau\right)=\left(fT \left.\right|\tau\right)=
\left(E_\Omega(fT)\left.\right|\tau\right)$ as required.

	(ii) This follows from the fact that for each $n\in\mathbb{N}$ and $K\Subset\Omega$ we have $\sup_{x\in
K}|(g\circledast\varphi-g)(x)h(x)|<1/n$ and $\sup_{x\in
K}|(g\circledast\varphi)(x)(h\circledast\varphi-h)(x)|<1/n$ a.e. in $\mathcal{D}_0$ (Definition~\ref{D:
Almost Everywhere}) which can be seen by elementary observation.
\end{proof} 

	Let  $\Omega, \Omega^\prime\in\mathcal{T}^d$ and 
$\psi\in\Diff(\Omega^\prime, \Omega)$. Then $E_\Omega(T)\circ\psi
= E_{\Omega^\prime}(T\circ\psi)$ does not generally hold in
$\widehat{\mathcal{E}(\Omega)^{\mathcal{D}_0}}$. That means that the family
of algebras $\{\widehat{\mathcal{E}(\Omega)^{\mathcal{D}_0}}\}_{\Omega\in\mathcal{T}^d}$ is not
diffeomorphism invariant (see the Introduction). Here $T\circ\psi$ stands for the composition in the sense of the
distribution theory (Vladimirov~\cite{vVladimirov}). Instead, we have the following weaker result. 
\begin{theorem}[Diffeomorphisms]\label{T: Diffeomorphisms} $E_\Omega$ weakly preserves the composition
with diffeomorphisms in the sense that for every $\Omega, \Omega^\prime\in\mathcal{T}^d$, every
$T\in\mathcal{D}^\prime(\Omega)$ and every
$\psi\in\Diff(\Omega^\prime, \Omega)$ we have  $E_\Omega(T)\circ\psi
\cong E_{\Omega^\prime}(T\circ\psi)$, i.e. $\left(E_\Omega(T)\circ\psi
\left.\right|\tau\right)=\left(E_{\Omega^\prime}(T\circ\psi)
\left.\right|\tau\right)$ for all test functions $\tau\in\mathcal{D}(\Omega^\prime)$. 
\end{theorem}
\begin{proof}  The proof is analogous to the proof of part (i) of Theorem~\ref{T: Weak Preservation}
and we leave the details to the reader.
\end{proof} 
\begin{examples} 
\begin{enumerate}
\item Let $\delta\in\mathcal{D}^\prime(\mathbb{R}^d)$ be the Dirac delta function
(delta distribution) on $\mathbb{R}^d$. For  its $\varphi$-regularization (\#\ref{No: phi-regularization} in
Examples~\ref{Exs: Nets and Distributions}) we have
$\delta_\varphi=\delta\conv\varphi=\delta\star\varphi=\varphi$. Thus 
$E_{\mathbb{R}^d}(\delta)=\widehat{\varphi}$.  Similarly,
$E_{\mathbb{R}^d}(\partial^\alpha\delta)=\widehat{\partial^\alpha\varphi}$.
\item We have
$\left(E_{\mathbb{R}^d}(\delta)\right)^n=(\widehat{\varphi})^n=\widehat{\varphi^n},\, n=1, 2,\dots$. We
express this result simply as
$\delta^n=\widehat{\varphi^n}$. Recall that 
the powers $\delta^n$ are meaningless within $\mathcal{D}^\prime(\mathbb{R}^d)$ for
$n\geq2$.
\item Let $H(x)$ be the Heaviside step function on $\mathbb{R}$. For  its $\varphi$-regularization 
(\#\ref{No: Loc} in Examples~\ref{Exs: Nets and Distributions}) we have
$H_\varphi=(H\conv\varphi)$. Let $K\Subset\mathbb{R}$. We observe that for every $x\in K$ we have 
$H_\varphi(x)=(H\star\varphi)(x)=\int_{-\infty}^x\varphi(t)\, dt$ a.e.  in $\mathcal{D}_0$ (Definition~\ref{D:
Almost Everywhere}). Thus
$E_{\mathbb{R}}(H)=\widehat{\int_{-\infty}^x\varphi(t)\, dt}$.  We express this result simply
as  $H(x)=\widehat{\int_{-\infty}^x\varphi(t)\, dt}$. Since
the embedding
$E_\mathbb{R}$ preserves the differentiation, we have $H^\prime=\delta$. 
\item We have
$E_{\mathbb{R}}(H)E_{\mathbb{R}}(\delta)=\widehat{\varphi}\left(\widehat{\int_{-\infty}^x\varphi(t)\,
dt}\right)=\widehat{\varphi H_\varphi}$. We express this result simply as
$H\delta=\widehat{\varphi H_\varphi}$. Recall that the product
$H\delta$ is not meaningful within $\mathcal{D}^\prime(\mathbb{R})$. 
\item We have
$\left(E_{\mathbb{R}}(H)\right)^n=\left(\widehat{\int_{-\infty}^x\varphi(t)\,
dt}\right)^n=\widehat{(H_\varphi)^n}$ which we write simply as $H^n=\widehat{(H_\varphi)^n}$.
Since $\widehat{\mathcal{E}(\mathbb{R})^{\mathcal{D}_0}}$ is a differential algebra, we can apply the
chain rule: $(H^n)^\prime=nH^{n-1}\delta$ which also is meaningless in
$\mathcal{D}^\prime(\mathbb{R})$ for $n\geq2$. 
\item Notice that $H^n\not=H,\, n=2, 3,\dots$ in
$\widehat{\mathcal{E}(\mathbb{R})^{\mathcal{D}_0}}$. Actually
$H^n=H,\, n=2, 3,\dots$, fail in any differential algebra. Indeed,  $H^2=H$ implies $2H\delta=\delta$ while 
$H^3=H$ implies $3H\delta=\delta$ thus $2=3$, a contradiction.  For a discussion we refer to 
(Grosser, Farkas, Kunzinger \& Steinbauer~\cite{mGrosser
at al 2}, Example (1.1.1)).
\end{enumerate}
\end{examples}
\section{Distributional Non-Standard Model}\label{S: Distributional Non-Standard Model}

	  \quad The {\em distributional non-standard model}
presented in this section is especially designed for the purpose of the {\em non-linear theory of generalized
functions} ({\em Colombeau theory}). It is a
$\mathfrak{c}^+$-saturated ultrapower non-standard model  with the set of individuals $\mathbb{R}$ based
on the $\mathcal{D}_0$-nets (Definition~\ref{D: Index Set and Nets}). 	Here
$\mathfrak{c}=\card(\mathbb{R})$ and $\mathfrak{c}^+$ stands for the successor of
$\mathfrak{c}$. The connection of the theory of asymptotic numbers and functions (Section~\ref{S:
Asymptotic Numbers and Asymptotic Functions}) with non-standard analysis will be discussed in the next section. We should mention that a similar ultrapower non-standard model (with the same index set and
different ultrafilter) was used in Guy Berger's thesis~\cite{gBerger05} for studying delta-like solutions of Hopf's equation.

 For readers who are familiar with non-standard analysis this section
is a short review of the ultra-power approach to non-standard analysis introduced by  W. A. J.
Luxemburg~\cite{wLuxNotes} almost 40 years ago (see also Stroyan \& Luxemburg~\cite{StroLux76}). For
the reader without background in non-standard analysis, this section offers a short introduction to
the subject. For additional reading, we refer to Davis~\cite{mDavis}, Lindstr\o
m~\cite{tLin} and Chapter 2 in Capi\'{n}ski \& Cutland~\cite{CapinskiCutland95}. 

\begin{definition}[Distributional Non-Standard Model]\label{D: Distributional Non-Standard Model} 
\begin{enumerate}
\item\label{No: Superstructure} Let $S$ be an infinite set.  The  \textbf{superstructure on}
$S$  is defined by $V(S)=:\bigcup_{n=0}^\infty V_n(S)$, where $V_0(S)=S$ and 
$V_{n+1}(S)=V_n(S)\cup\mathcal{P}\left(V_n(S)\right)$.  The
\textbf{level} $\lambda(A)$ of $A\in V(S)$ is defined by the formula
$\lambda(A)=: \min\{n\in\mathbb{N}_0 : A\in V_n(S)\}$. The superstructure $V(S)$ 
is {\em transitive} in the sense that $V(S)\setminus S\subset\mathcal{P}(V(S))$. Thus $V(S)\setminus S$ is
a Boolean algebra. The members $s$ of $S$ are called {\em individuals} of the superstructure $V(S)$.
\item Let $S=\mathbb{R}$. We observe that $V(\mathbb{R})$ contains all objects in \textbf{standard
analysis}: all ordered pairs of real numbers thus the set of
complex numbers $\mathbb{C}$, Cartesian products of subsets of
$\mathbb{R}$ and of $\mathbb{C}$ thus all relations on $\mathbb{R}$ and on $\mathbb{C}$, all binary
algebraic operations on $\mathbb{R}$ and on $\mathbb{C}$, all real and complex functions, all
sets of functions, etc.

\item\label{No: *R} Let $\mathbb{R}^{\mathcal{D}_0}$ be the set of all $\mathcal{D}_0$-nets in
$\mathbb{R}$ (Definition~\ref{D: Index Set and Nets}). The set $^*\mathbb{R}$ of non-standard real
numbers is defined as follows:
\begin{description}
\item[(a)]
 We define the \textbf{equivalence relation}  $\sim_\mathcal{U}$ on $\mathbb{R}^{\mathcal{D}_0}$ 
by $(A_\varphi)\sim_\mathcal{U}(B_\varphi)$ if $A_\varphi=B_\varphi$ a.e. or, equivalently, if 
$\{\varphi\in\mathcal{D}_0 :  A_\varphi=B_\varphi\}\in\mathcal{U}$ (Definition~\ref{D: Almost Everywhere}).

\item[(b)] The equivalence classes in $^*\mathbb{R}=\mathbb{R}^{\mathcal{D}_0}/\sim_\mathcal{U}$ are
called \textbf{non-standard real numbers}. We denote by
$\left<A_\varphi\right>\in{^*\mathbb{R}}$ the equivalence class of  the net $(A_\varphi)\in
\mathbb{R}^{\mathcal{D}_0}$. The ring operations in $^*\mathbb{R}$ are inherited from the ring
$\mathbb{R}^{\mathcal{D}_0}$. The order in $^*\mathbb{R}$ is defined by $\left<A_\varphi\right>>0$ if 
$A_\varphi>0$ a.e., i.e. if $\{\varphi\in\mathcal{D}_0 : A_\varphi>0\}\in\mathcal{U}$.

\item[(c)] We define the \textbf{canonical embedding} $\mathbb{R} \emb{^*\mathbb{R}}$ by the constant
nets, i.e. by $A\to\left< A_\varphi\right>$, where $A_\varphi=A$ for all
$\varphi\in {\mathcal{D}_0}$. We shall write simply 
$\mathbb{R}\subseteq{^*\mathbb{R}}$ instead of $\mathbb{R}\emb{^*\mathbb{R}}$. Also if $(A_\varphi)$
is a constant net, we shall write simply  
$\left<A\right>$ instead of $\left< A_\varphi\right>$.
\end{description}
\item Let $S={^*\mathbb{R}}$. The superstructure $V(^*\mathbb{R})$ contains all objects in
\textbf{non-standard analysis}: ordered pairs of non-standard real numbers thus the set of  non-standard
complex numbers $^*\mathbb{C}$, all Cartesian products of subsets of
$^*\mathbb{R}$ and of $^*\mathbb{C}$ thus all relations on $^*\mathbb{R}$ and on $^*\mathbb{C}$,
all binary algebraic operations on $^*\mathbb{R}$ and on $^*\mathbb{C}$, all non-standard functions,
all sets of non-standard functions, etc.
\item Let $V(\mathbb{R})^{\mathcal{D}_0}$
stand for the set of all $\mathcal{D}_0$-nets in $V(\mathbb{R})$ (Definition~\ref{D: Index Set
and Nets}). A net
$(A_\varphi)$ in
$V(\mathbb{R})^{\mathcal{D}_0}$ is called \textbf{tame} if 
$(\exists n\in\mathbb{N}_0)(\forall\varphi\in\mathcal{D}_0)(A_\varphi\in V_n(\mathbb{R}))$.
If $(A_\varphi)$ is a tame net in $V(\mathbb{R})^{\mathcal{D}_0}$ its level $\lambda((A_\varphi))$ is
defined (uniquely) as the number $n\in\mathbb{N}_0$ such that
$\{\varphi\in\mathcal{D}_0: \lambda(A_\varphi)=n\}\in\mathcal{U}$, where 
$\lambda(A_\varphi)$ is the level of $A_\varphi$ in $V(\mathbb{R})$ (see \#\ref{No: Superstructure} above). 

\item For every tame net $(A_\varphi)$ in
$V(\mathbb{R})^{\mathcal{D}_0}$ we define $\left<A_\varphi\right>\in V(^*\mathbb{R})$ inductively on
the level of the nets: If $\lambda((A_\varphi))=0$,  then  $\left<A_\varphi\right>$ is defined
in \#\ref{No: *R} above.  Suppose $\left<A_\varphi\right>$ is already defined for all tame nets $(A_\varphi)$ in
$V(\mathbb{R})^{\mathcal{D}_0}$ with $\lambda((A_\varphi))<n$. If
$(B_\varphi)\in V(\mathbb{R})^{\mathcal{D}_0}$ is a tame net with
$\lambda((B_\varphi))=n$, we let $\left<B_\varphi\right>=: \left\{ (A_\varphi)\in
V(\mathbb{R})^{\mathcal{D}_0} : \, \lambda((A_\varphi))<n\, \&\, A_\varphi\in B_\varphi
\mbox{ a.e.}\right\}$, where, as before,
$A_\varphi\in B_\varphi \mbox{ a.e.}$ means $\left\{\varphi\in\mathcal{D}_0 :
A_\varphi\in B_\varphi\right\}\in\mathcal{U}$ (Definition~\ref{D: Almost Everywhere}).
Let $(A_\varphi)$ be a \textbf{constant net} in $V(\mathbb{R})^{\mathcal{D}_0}$, i.e. 
$A_\varphi=A$ for all $\varphi\in\mathcal{D}_0$ and some $A\in V(\mathbb{R})$. In the case of 
constant nets we shall write simply $\left<A\right>$ instead of $\left<A_\varphi\right>$.

\item\label{No: Internal} An element $\mathcal{A}$ of $V(^*\mathbb{R})$ is called
\textbf{internal} if  $\mathcal{A}=\left<A_\varphi\right>$ for some tame net  $(A_\varphi)\in
V(\mathbb{R})^{\mathcal{D}_0}$.  We denote by $^*V(\mathbb{R})$ the
\textbf{set of the internal elements} of $V(^*\mathbb{R})$ (including the non-standard reals in
$^*\mathbb{R}$). The elements of   $^*V(\mathbb{R})\setminus{^*\mathbb{R}}$ are called
\textbf{internal sets}. The internal sets of the form 
$\left<A\right>$, where  $A\in V(\mathbb{R})$ (i.e. generated by constant nets), are 
called \textbf{internal standard} (or simply, {\em standard}). The elements of
$V(^*\mathbb{R})\setminus{^*V(\mathbb{R})}$ are called \textbf{external sets}.

\item\label{No: Extension Mapping} We define the \textbf{extension mapping}   $*: V(\mathbb{R})\to
V(^*\mathbb{R})$ by $^*A=\left<A\right>$. Notice that the range $\ran(*)$ of the extension mapping $*$
consists exactly of the internal standard elements of $V(^*\mathbb{R})$. The terminology {\em extension
mapping} for $*$ is due to the following result: Let $S\in V(\mathbb{R})\setminus\mathbb{R}$. Then
$S\subseteq{^*S}$ and the equality occurs \ifff $S$ is a finite set.

\item It can be shown that $\mathcal{A}$ is internal \ifff  $\mathcal{A}\in{^*A}$ for some $A\in
V(\mathbb{R})$. It can be shown as well that an element
$A\in V(\mathbb{R})$ is internal \ifff $A\in{\mathbb{R}}$ or $A$ is a finite set (notice that
$V(\mathbb{R})\subseteq V({^*\mathbb{R}})$ since $\mathbb{R}\subseteq{^*\mathbb{R}}$).  The infinite
sets in $V(\mathbb{R})\setminus\mathbb{R}$ are called \textbf{external standard sets}. For example, the
familiar $\mathbb{N}, \mathbb{N}_0, \mathbb{Z}, \mathbb{Q}, \mathbb{R}, \mathbb{C}$ are
all external standard sets.
\item\label{No: Infinitesimals in *C}   A point $\zeta\in{^*\mathbb{C}^d}$ is called \textbf{infinitesimal} if
$||\zeta||<1/n$ for all
$n\in\mathbb{N}$.  Also, $\zeta\in{^*\mathbb{C}^d}$ is called \textbf{finite} if $||\zeta||<n$ for some
$n\in\mathbb{N}$. Similarly,  $\zeta\in{^*\mathbb{C}^d}$ is called \textbf{infinitely large} if $n<||\zeta||$ for
all $n\in\mathbb{N}$. We denote by
$\mathcal{I}(^*\mathbb{C}^d), \mathcal{F}(^*\mathbb{C}^d)$ and $\mathcal{L}(^*\mathbb{C}^d)$ 
the sets of the \textbf{infinitesimal, finite and infinitely large points} in
$^*\mathbb{C}^d$, respectively. We often write $\zeta\approx 0$ instead of
$\zeta\in\mathcal{I}(^*\mathbb{C}^d)$ and $\zeta_1\approx \zeta_2$ instead of
$\zeta_1-\zeta_2\in\mathcal{I}(^*\mathbb{C}^d)$. More generally, if
$\mathcal{S}\subseteq{^*\mathbb{C}^d}$, then $\mathcal{I}(\mathcal{S}),
\mathcal{F}(\mathcal{S})$ and $\mathcal{L}(\mathcal{S})$ denote the sets of infinitesimal, finite and
infinitely large points in $\mathcal{S}$, respectively.
\item We define the \textbf{standard part mapping}  $\st:
\mathcal{F}(^*\mathbb{C}^d)\to
\mathbb{C}^d$ by the formula $\st(\zeta)\approx \zeta$. We observe that $\st$ is a vector
homomorphism from 
$\mathcal{F}(^*\mathbb{C}^d)$ onto
$\mathbb{C}^d$. In particular, $\st: \mathcal{F}(^*\mathbb{C})\to\mathbb{C}$ is an order preserving ring
homomorphism from 
$\mathcal{F}(^*\mathbb{C})$ onto $\mathbb{C}$ (relative to the partial order in $^*\mathbb{C}$).

\item\label{No: Canonical Infinitesimal} We call $\rho\in{^*\mathbb{R}}$, defined by 
$\rho=\left<R_\varphi\right>$ (cf. (\ref{E: RadiusSupport})), the \textbf{canonical
infinitesimal} in $^*\mathbb{R}$. It is {\em canonical} because is defined uniquely in terms of the index set
of the distributional non-standard model. It is a {\em positive infinitesimal} because
$0<\rho<1/n$ for all
$n\in\mathbb{N}$ (Example~\ref{Ex: Radius of Support}). 

\item\label{No: Monad} Let  $x\in{\mathbb{R}^d}$ and $X\subseteq{\mathbb{R}^d}$. The \textbf{monads}
of  $x$ and $X$ are defined by 
\begin{align}
&\mu(x)=\left\{x+dx : dx\in{^*\mathbb{R}^d}\; \&\, ||dx||\approx 0
\right\},\notag\\
&\mu(X)=\left\{x+dx : x\in X\; \&\; dx\in{^*\mathbb{R}^d}\; \&\; ||dx||\approx 0
\right\},\notag
\end{align}
respectively. Also, $\mu_0(x)=: \mu(x)\setminus\{x\}$ is the \textbf{deleted monad} of $x$.
\end{enumerate}
\end{definition}

\begin{theorem}[Extension Principle]\label{T: Extension Principle} $^*\mathbb{R}$ is a proper
extension of $\mathbb{R}$, i.e. $\mathbb{R}\subsetneqq{^*\mathbb{R}}$.
Consequently, $V(\mathbb{R})\subsetneqq V({^*\mathbb{R}})$.
\end{theorem}
\begin{proof}  We observe that $\rho\in{^*\mathbb{R}}\setminus\mathbb{R}$ (\#\ref{No: Canonical
Infinitesimal} in Definition~\ref{D: Distributional Non-Standard Model}).
\end{proof} 

	  In what follows we assume a particular case of the continuum hypothesis in the form 
$\mathfrak{c}^+=2^\mathfrak{c}$.

\begin{theorem}[Saturation Principle]\label{T: Saturation Principle} Our non-standard model
$V(^*\mathbb{R})$ is $\mathfrak{c}^+$-\textbf{saturated} in the sense that every family
$(\mathcal{A}_{\gamma})_{\gamma\in\Gamma}$ of internal sets in $V(^*\mathbb{R})$
with the finite intersection property  and $\card({\Gamma})\leq \mathfrak{c}$ has the non-empty intersection
 $\bigcap_{\gamma\in\Gamma}\mathcal{A}_{\gamma}\not=\varnothing$.
Also  $V(^*\mathbb{R})$ is \textbf{fully saturated} in the sense that $V(^*\mathbb{R})$ is
$\card(^*\mathbb{R})$-saturated (cf. Chang {\em \&}
Keisler~\cite{CKeis}, Chpter 5).
\end{theorem}
\begin{proof}  We refer the reader to the original proof in Chang \& Keisler~\cite{CKeis}
(for a presentation see also  Lindstr\o m~\cite{tLin}). We should mention that the
property of the ultrafilter  $\mathcal{U}$ to be $\mathfrak{c}^+$-good (\# \ref{No: Good} in Lemma~\ref{L:
List of Properties of U}) is involved in the proof of this theorem. To show that $V(^*\mathbb{R})$ is fully
saturated, we have to show that
$\card(^*\mathbb{R})=\mathfrak{c}^+$.  Indeed, 
$\card{(^*\mathbb{R})}\leq\card(\mathbb{R}^{\mathcal{D}_0})=2^\mathfrak{c}$ follows from the definition
of $^*\mathbb{R}$ in the distributional model and $\card(^*\mathbb{R})\geq2^\mathfrak{c}$ follows from the
fact that $V(^*\mathbb{R})$ is $\mathfrak{c}^+$-\textbf{saturated}.
\end{proof} 
\begin{theorem}[Order Completeness Principle]\label{T: Order Completeness Principle} Let
$\mathcal{A}$ be an internal non-empty bounded from above subset of $^*\mathbb{R}$. Then
$\sup(\mathcal{A})$ exists in $^*\mathbb{R}$. Also, if $\sup(\mathcal{A})$
exists, then there exists a net $(\mathcal{A}_\varphi)\in\mathcal{P}(\mathbb{R})^{\mathcal{D}_0}$
(Examples~\ref{Exs: Nets and Distributions}) such that
$\mathcal{A}=\left<\mathcal{A}_\varphi\right>$ and
$\sup(\mathcal{A})=\left<\sup(\mathcal{A}_\varphi)\right>$.
\end{theorem}
\begin{proof} For the proof we refer to (Lindstr\o m~\cite{tLin}, p. 11).
\end{proof}
\begin{theorem}[Spilling Principles]\label{T: Spilling Principles} Let
$\mathcal{A}\subseteq{^*\mathbb{R}}$ be an internal set. Then:
\begin{description}

\item{\bf (i) Overflow of $\mathcal{F}(^*\mathbb{R})$:} If $\mathcal{A}$ contains arbitrarily
large finite numbers, then $\mathcal{A}$ contains arbitrarily small infinitely large numbers.
\item{\bf (ii) Underflow of $\mathcal{F}(^*\mathbb{R})\setminus\mathcal{I}(^*\mathbb{R})$:} 
If $\mathcal{A}$ contains arbitrarily small finite non-infinitesimals, then $\mathcal{A}$
contains arbitrarily large infinitesimals. 
\item{\bf (iii) Overflow of $\mathcal{I}(^*\mathbb{R})$:} If $\mathcal{A}$ contains arbitrarily
large infinitesimals, then $\mathcal{A}$ contains arbitrarily small finite non-infinitesimals.
\item{\bf (iv) Underflow of $\mathcal{L}(^*\mathbb{R})$:} If $\mathcal{A}$
contains arbitrarily small infinitely large numbers, then
$\mathcal{A}$ contains arbitrarily large finite numbers. 
\end{description}
\end{theorem}
\begin{proof}
	For the proof we refer to Davis~\cite{mDavis} or Lindstr\o m~\cite{tLin}.
\end{proof}
	 The next result demonstrates the remarkable feature of non-standard analysis to reduce (and
sometimes even to eliminate completely) the number of quantifiers compared with standard analysis.
\begin{theorem}[Usual Topology on $\mathbb{R}^d$ and Monads]\label{T: Usual Topology on Rd and
Monads} Let $X\subseteq{\mathbb{R}^d}$ and $x\in{\mathbb{R}^d}$. Then: {\bf (a)} $x$ is an
\textbf{interior point} of $X$ \ifff $\mu(x)\subseteq{^*X}$. Consequently, $X$ is \textbf{open} \ifff
$\mu(X)\subseteq{^*X}$. {\bf (b)} $X$ is \textbf{closed} \ifff $\st({^*X})=X$, where $\st:
\mathcal{F}(^*\mathbb{R}^d)\to\mathbb{R}^d$ stands for the standard part mapping. {\bf (c)}
$x$ is an \textbf{adherent point} of $X$ (i.e. $x\in\overline{X}$) \ifff $^*X\cap\mu(x)\not=\varnothing$.
{\bf (d)} $X$ is a \textbf{cluster point} of $X$ \ifff
$^*X\cap\mu_0(x)\not=\varnothing$. {\bf (e)} $X$ is a \textbf{bounded set} \ifff $^*X$ consists of finite
points only. {\bf (f)}  $X$ is \textbf{compact} \ifff $^*X\subseteq\mu(X)$. 
\end{theorem}
\begin{proof}    We refer the reader to the original proofs in Robinson~\cite{aRob66} (or, to a
presentation  in Salbany \& Todorov~\cite{SalbTod98}).
\end{proof} 

	 The complete our survey on non-standard analysis we have to discuss two more important
principles: the {\em transfer principle} and {\em internal definition principle}. The transfer principle is
considered by many as the ``heart and soul of non-standard analysis''. The formulation of these two
principles however requires a more precise choice of our formal language. The reader who do not have taste
for mathematical logic might skip (or browse casually through) the rest of this section.

\begin{definition}[Formal Language]\label{D: Formal Language} Let $S$ be (as before) an infinite set and $V(S)$
be the superstructure on $S$. The formal \textbf{language} $\mathcal{L}\mathcal{A}\mathcal{N}(S)$ 
with {\em set of individuals} $S$ is
constructed as follows:
\begin{enumerate}
\item \label{No: alphabet} The \textbf{alphabet} of $\mathcal{L}\mathcal{A}\mathcal{N}(S)$ consists of
the  three mutually disjoint sets $\mathcal{A}\cup\mathcal{B}\cup V(S)$, where

\begin{description}
\item[(a)] $\mathcal{A}=\left\{=, \in, \neg, \wedge\,  \mbox{(and) }, \vee\, 
\mbox{(or) },\forall, \exists, \Rightarrow, \Leftrightarrow, (), [\;], \{\, \}\right\}$ is the (finite) set of
{\em symbols}.
\item[(b)] $\mathcal{B}=\left\{x, y, z, X, Y, Z, x_1, x_2,\dots\right\}$ is the countable set of  {\em variables}.
\end{description}

The members of $\mathcal{A}\cup\mathcal{B}\cup V(S)$ are called {\em letters}.  A {\em word} is any
finite string of letters.

\item The \textbf{vocabulary} of the
language $\mathcal{L}\mathcal{A}\mathcal{N}(S)$ consists of: {\em words, terms, predicates and
propositions}. They are defined recursively (by the length of the word), where all quantifiers in the predicates
and propositions are bounded by elements in $V(S)\setminus S$. In other words,  
$\mathcal{L}\mathcal{A}\mathcal{N}(S)$ allows only predicates (propositions) such as
$(\forall x_1\in A)P(x_1, x_2\dots x_n)$ or $(\forall x_1\in A)(\exists x_2\in B)P(x_1, x_2\dots x_n)$, 
where $A, B\in V(S)\setminus S$. Here $P(x_1, x_2\dots x_n)$ stands for a predicate in the  free
variables $x_1,\dots x_n\in\mathcal{B}$. Notice that the language
$\mathcal{L}\mathcal{A}\mathcal{N}(S)$ disallows predicates (propositions) such as  $(\forall x_1)P(x_1,
x_2\dots x_n)$ or
$(\forall x_1)(\exists x_2)P(x_1, x_2\dots x_n)$. 

\item We supply the set of propositions in the language $\mathcal{L}\mathcal{A}\mathcal{N}(S)$ with the
usual
\textbf{semantics} (true or false values) inherited from the Boolean structure of $V(S)\setminus S$.

\end{enumerate}
\end{definition}

\begin{examples} Here are our two basic examples:
\begin{enumerate}
\item \label{No: StandardFramework} Let $S=\mathbb{R}$. Then 
$\mathcal{L}\mathcal{A}\mathcal{N}(\mathbb{{R}})$ is the
\textbf{formal language of standard analysis}. 
\item\label{No: NonStandardFramework} Let $S={^*\mathbb{R}}$. Then
$\mathcal{L}\mathcal{A}\mathcal{N}(^*\mathbb{R})$  is the
\textbf{formal language of non-standard analysis}. 
\end{enumerate}
\end{examples}

	 For more details on the topic we refer to Davis~\cite{mDavis}, Lindstr\o m~\cite{tLin} and Chapter 2 in
Capi\'{n}ski \& Cutland~\cite{CapinskiCutland95}. We believe however that
the reader can successfully proceed to the rest of this article without more specialized knowledge in
logic. 
	
\begin{theorem}[Transfer Principle]\label{T: Transfer Principle} Let $P(x_1,\dots, x_q)$  be a predicate
in the language
$\mathcal{L}\mathcal{A}\mathcal{N}(\mathbb{{R}})$ and let $A_n\in V(\mathbb{R}),  n=1, 2,\dots,
q$. Let $P({A}_1,\dots, {A}_q)$ and $P({^*A}_1,\dots, {^*A}_q)$ be the propositions in the
languages $\mathcal{L}\mathcal{A}\mathcal{N}(\mathbb{{R}})$ and
$\mathcal{L}\mathcal{A}\mathcal{N}(^*\mathbb{{R}})$, respectively, obtained from $P(x_1, \dots, x_q)$ by
replacing all
$x's$ by 
$A's$ or $^*A's$, respectively. Then
$P({A}_1,\dots, {A}_q)$ and $P({^*A}_1,\dots,
{^*A}_q)$ are equivalent, i.e. $P({A}_1,\dots, {A}_q)\Leftrightarrow P({^*A}_1,\dots,
{^*A}_q)$.
\end{theorem}
\begin{proof}  We refer the reader to  Davis~\cite{mDavis}
(for a presentation, see also Lindstr\o m~\cite{tLin}). 
\end{proof} 

\begin{corollary}[Field Properties]\label{C: Field Properties} $^*\mathbb{R}$ is a \textbf{non-archimedean
real} \textbf{closed} (thus, totally ordered)
\textbf{field}. Consequently, $^*\mathbb{C}$ is a \textbf{non-archimedean algebraically closed field} and
we  have the usual connection $^*\mathbb{C}={^*\mathbb{R}}(i)$, where
$i=\sqrt{-1}$.
\end{corollary}
\begin{proof}   The field properties follow from the transfer principle. For example, the proposition 
$(\forall x\in{^*\mathbb{R}})[x\not= 0\Rightarrow (\exists y\in{^*\mathbb{R}})(xy=1)]$ is true in
$\mathcal{L}\mathcal{A}\mathcal{N}(^*\mathbb{{R}})$ because $(\forall x\in{\mathbb{R}})[x\not=
0\Rightarrow (\exists y\in{\mathbb{R}})(xy=1)]$ is true in
$\mathcal{L}\mathcal{A}\mathcal{N}(\mathbb{{R}})$ and similarly for the rest of the {\em real closed field
axioms} and {\em algebraically closed field
axioms}. Also both $^*\mathbb{R}$ and $^*\mathbb{C}$ are
non-archimedean because they are proper extensions of $\mathbb{R}$ and
$\mathbb{C}$, respectively. 
\end{proof} 
\begin{remark}[An Alternative Proof] The above corollary can also be
proved  without  transfer principle. We have to involve the nets in
$\mathbb{R}^{\mathcal{D}_0}$ and $\mathbb{C}^{\mathcal{D}_0}$  and use the properties of the ultrafilter
$\mathcal{U}$ listed in Lemma~\ref{L: List of Properties of U}. This second proof is very similar to the proof of
Theorem~\ref{T: Algebraic Properties}. 
\end{remark} 

\begin{theorem}[Internal Definition Principle]\label{T: Internal Definition Principle} Let
$\mathcal{A}\in{^*V(\mathbb{R})}\setminus{^*\mathbb{R}}$ (be an internal set) and let
$\mathcal{A}_n\in{^*V(\mathbb{R})}, n=1,\dots,q$ (be non-standard real numbers or internal sets). Let
$P(x, x_1,\dots, x_q)$ be a predicate in
$q+1$ variables in the language $\mathcal{L}\mathcal{A}\mathcal{N}(\mathbb{{R}})$ and let ${P(x,
\mathcal{A}_1,\dots,
\mathcal{A}_q)}$ be the corresponding predicate in a single variable in the language
$\mathcal{L}\mathcal{A}\mathcal{N}(^*\mathbb{{R}})$. Then the set  $\mathcal{B}=:\{x\in\mathcal{A} :
P(x, \mathcal{A}_1,\dots, \mathcal{A}_q)\}$ is also internal.
\begin{proof} We refer the reader to Davis~\cite{mDavis}.
\end{proof}
\end{theorem}
\begin{remark}[Axiomatic Approach]\label{R: Axiomatic Approach}  The extension, saturation and
transfer principles are  theorems in the distributional model presented above. In one of the  axiomatic
approaches to non-standard analysis however these three principles are treated as axioms. For a discussion we
refer to  (Lindstr\o m~\cite{tLin}, pp. 81-83 and pp. 97-98).  
\end{remark}

\section{J.F. Colombeau's Theory of Generalized Functions and Non-Standard
Analysis}\label{S: J.F. Colombeau's Non-Linear Theory of Generalized Functions and Non-Standard
Analysis}

	\quad We show that  the field of asymptotic numbers
$\widehat{\mathbb{C}^{\mathcal{D}_0}}$ (Definition~\ref{D: Asymptotic Numbers}) is isomorphic to a
particular Robinson field ${^\rho\mathbb{C}}$ (Robinson~\cite{aRob73}) of $\rho$-asymptotic
numbers. We also prove that the algebra of asymptotic functions
$\widehat{\mathcal{E}(\Omega)^{\mathcal{D}_0}}$ (Definition~\ref{D: Asymptotic Functions}) is
isomorphic to a particular algebra of
$\rho$-asymptotic functions $^\rho\mathcal{E}(\Omega)$ introduced in (Oberguggenberger \&
Todorov\cite{OberTod98}). Both
${^\rho\mathbb{C}}$ and
$^\rho\mathcal{E}(\Omega)$ are defined in the framework of non-standard analysis 
(see Definition~\ref{D: A. Robinson's rho-Asymptotic Numbers} and Definition~\ref{D: rho-Asymptotic
Functions} below). As far as we treat
$\widehat{\mathbb{C}^{\mathcal{D}_0}}$ and $\widehat{\mathcal{E}(\Omega)^{\mathcal{D}_0}}$ as
modified and, we believe, improved versions of Colombeau's $\overline{\mathbb{C}}$ and
$\mathcal{G}(\Omega)$, respectively, these results establish a
connection between Colombeau theory and non-standard analysis. 

	 Recall the definition of A. Robinson's field ${^\rho\mathbb{R}}$ (Robinson~\cite{aRob73}) and its
complex counterpart ${^\rho\mathbb{C}}$.

\begin{definition}[Robinson $\rho$-Asymptotic Numbers]\label{D: A. Robinson's rho-Asymptotic
Numbers} Let $^*\mathbb{R}$ and $^*\mathbb{C}$ be the
non-standard extensions of $\mathbb{R}$ and $\mathbb{C}$, respectively in an arbitrary
$\kappa$-saturated non-standard model with set of individuals $\mathbb{R}$, where $\kappa$ is an infinite
cardinal. (In particular, this could be
the distributional non-standard model constructed in Section~\ref{S: Distributional Non-Standard Model}).
Let $\rho$ be a positive infinitesimal in
${^*\mathbb{R}}$. Following Robinson~\cite{aRob73}, we define:
\begin{enumerate}
\item  The sets of the {\bf
$\rho$-moderate} and {\bf $\rho$-negligible} non-standard complex numbers are
\begin{align}\label{E: rho-moderate}
&\mathcal{M}_\rho(^*\mathbb{C})=\left\{\zeta\in{^*\mathbb{C}} : \; |\zeta|\leq\rho^{-m} \text{\: for
some\;} m\in\mathbb{N}\right\},\\ 
&\mathcal{N}_\rho(^*\mathbb{C})=\left\{\zeta\in{^*\mathbb{C}}:\;
|\zeta|<\rho^{n} \text{\: for all\;}
n\in\mathbb{N}\right\},
\end{align}
respectively. \textbf{Robinson field of complex $\rho$-asymptotic numbers} is the factor ring
${^\rho\mathbb{C}}=:
\mathcal{M}_\rho(^*\mathbb{C})/\mathcal{N}_\rho(^*\mathbb{C})$. We denote by $\widehat{\zeta}$ the
equivalence class of $\zeta\in\mathcal{M}_\rho(^*\mathbb{C})$.  For example,
$\widehat{\rho}$  is the asymptotic number corresponding to
$\rho$. 
\item If $\mathcal{S}\subseteq{^*\mathbb{C}}$, we let
$\widehat{S}=\{\widehat{\zeta} : \zeta\in{\mathcal{S}\cap\mathcal{M}_\rho(^*\mathbb{C})}\}$. 
If  $S\subseteq{\mathbb{C}}$, then
$^\rho\!S=:\widehat{^*S}$ is called the $\rho$-extension of $S$. In particular, the
field of \textbf{Robinson real $\rho$-asymptotic numbers} ${^\rho\mathbb{R}}$ is the
$\rho$-extension of $\mathbb{R}$, i.e. 
${^\rho\mathbb{R}}=\widehat{{^*\mathbb{R}}}$. We define an order relation in ${^\rho\mathbb{R}}$ as
follows:  Let $\widehat{\xi}\in{^\rho\mathbb{R}}$ and 
$\widehat{\xi}\not=0$. Then $\widehat{\xi}> 0$ if $\xi>0$ in $^*\mathbb{R}$.

\item We supply ${^\rho\mathbb{C}}$ with the \textbf{order topology}, i.e. the product topology inherited
from the order topology on ${^\rho\mathbb{R}}$.
\item \label{No: Ultra-Norm} The \textbf{valuation} $v:{^\rho\mathbb{C}}\to\mathbb{R}\cup\{\infty\}$  is
defined by
$v(0)=\infty$ and 
$v(\widehat{\zeta})={\rm st}\left(\ln{|\zeta|}/\ln{\rho}\right)$ if 
$\widehat{\zeta}\in{^\rho\mathbb{C}},\; \widehat{\zeta}\not=0$. We define an \textbf{ultra-norm}
$|\cdot|_v: {^\rho\mathbb{C}}\to\mathbb{R}$ by the formula $|z|_v=e^{-v(z)}$ (under the
convention that $e^{-\infty}=0$) and an {\bf ultra-metric} by  $d_v(a, b)=|a-b|_v$.

\item Let $\xi=(\xi_1,\dots, \xi_d)\in{^*\mathbb{R}^d}$ and $||\xi||\in{\mathcal{M}_\rho(^*\mathbb{C})}$.
We  define $\widehat{\xi}\in{^\rho\mathbb{R}^d}$ by 
$\widehat{\xi}=(\widehat{\xi_1},\dots, \widehat{\xi_d})$. Let
$\Omega$ be an open set of $\mathbb{R}^d$ and 
$\mu(\Omega)$ be the monad of
$\Omega$ (\#\ref{No: Monad}  in Definition~\ref{D: Distributional Non-Standard Model}). We denote
$\widehat{\mu(\Omega)}=\{\widehat{\xi} :
\xi\in\mu(\Omega)\}$.  
\end{enumerate}
\end{definition}

 The next result appears in (Lightstone \& Robinson~(\cite{LiRob}, p. 97).
\begin{theorem}[Principles of Permanence]\label{T; Principles of Permanence} Let
$\mathcal{A}\subseteq{^*\mathbb{R}}$ be an internal set. 
\begin{description}
\item{\bf (a) Overflow of $\mathcal{M}_\rho(^*\mathbb{R})$:} If $\mathcal{A}$ contains
arbitrarily large numbers in
$\mathcal{M}_\rho(^*\mathbb{R})$, then $\mathcal{A}$ contains arbitrarily small
numbers in
$^*\mathbb{R}\setminus\mathcal{M}_\rho(^*\mathbb{R})$. 
\item{\bf (b) Underflow of
$\mathcal{M}_\rho(^*\mathbb{R})\setminus\mathcal{N}_\rho(^*\mathbb{R})$:} If
$\mathcal{A}$ contains arbitrarily small numbers in
$\mathcal{M}_\rho(^*\mathbb{R})\setminus\mathcal{N}_\rho(^*\mathbb{R})$, then
$\mathcal{A}$ contains arbitrarily large numbers in
$\mathcal{N}_\rho(^*\mathbb{R})$. 
\item{\bf (c) Overflow of $\mathcal{N}_\rho(^*\mathbb{R})$:} If $\mathcal{A}$ contains
arbitrarily large numbers in\newline $\mathcal{N}_\rho(^*\mathbb{R})$, then $\mathcal{A}$
contains arbitrarily small numbers in
$\mathcal{M}_\rho(^*\mathbb{R})\setminus\mathcal{N}_\rho(^*\mathbb{R})$. 
\item{\bf (d) Underflow of $^*\mathbb{R}\setminus\mathcal{M}_\rho(^*\mathbb{R})$:} If
$\mathcal{A}$ contains arbitrarily small numbers in
$^*\mathbb{R}\setminus\mathcal{M}_\rho(^*\mathbb{R})$, then
$\mathcal{A}$ contains arbitrarily large numbers in $\mathcal{M}_\rho(^*\mathbb{R})$.
\end{description}
\end{theorem}

\begin{theorem}[Field Properties]\label{T: Field Properties}
$^\rho\mathbb{C}$ is an algebraically closed field, 
$^\rho\mathbb{R}$ is a real closed field and we have the usual connection
$^\rho\mathbb{C}={^\rho\mathbb{R}}(i)$. 
\end{theorem}
\begin{proof} The connection
$^\rho\mathbb{C}={^\rho\mathbb{R}}(i)$ follows directly from the definition of 
$^\rho\mathbb{C}$ and $^\rho\mathbb{R}$. The proof that
$^\rho\mathbb{R}$ is a field can be found in (Lightstone \& Robinson~\cite{LiRob}, p.78). It follows
that $^\rho\mathbb{C}$ is also a field. Let $P(x)=x^p+a_{p-1}x^{p-1}+\dots+a_0$ be a polynomial with
coefficients in $^\rho\mathbb{C}$ and a degree $p\geq 1$. We have $a_n=\widehat{\alpha_n}$ for some
$\alpha_n\in\mathcal{M}_\rho(^*\mathbb{C})$. We let
$Q(x)=x^p+\alpha_{p-1}x^{p-1}+\dots+\alpha_0$. Next,  we observe that $^*\mathbb{C}$ is an
algebraically closed field by transfer principle (cf. Theorem~\ref{T: Transfer Principle} in this article or Davis~\cite{mDavis}) since $\mathbb{C}$ is an
algebraically closed field. Thus the equation
$Q(\zeta)=0$ has a solution
$\zeta$ in ${^*\mathbb{C}}$. The estimation
$|\zeta|\leq 1+|\alpha_{p-1}|+\dots+|\alpha_0|$ shows that $\zeta\in\mathcal{M}_\rho(^*\mathbb{C})$. 
Thus $P(\widehat{\zeta})=\widehat{Q(\zeta)}=\widehat{0}=0$ proving that  $^\rho\mathbb{C}$ is an
algebraically closed field. It follows that  $^\rho\mathbb{R}$ is a real closed field as a maximal real subfield of
$^\rho\mathbb{C}$ (Van Der Waerden~\cite{VanDerWaerden}, Chapter 11).
\end{proof}

	 We turn to the connection between Robinson's theory of the field $^\rho\mathbb{R}$ and the field of
asymptotic numbers defined in Definition~\ref{D: Asymptotic Numbers}.
\begin{theorem}[Isomorphic Fields]\label{T: Isomorphic Fields} Let $^*\mathbb{R}$ and $^*\mathbb{C}$ be
the non-standard extensions of $\mathbb{R}$ and $\mathbb{C}$, respectively (\#\ref{No: *R} and \#\ref{No:
Extension Mapping}, Definition~\ref{D: Distributional Non-Standard Model}) defined within
our distributional  non-standard  model (Section~\ref{S: Distributional Non-Standard
Model}). Let $\rho=\left<R_\varphi\right>$ be the canonical infinitesimal  in
${^*\mathbb{R}}$ (\#\ref{No: Canonical Infinitesimal} in Definition~\ref{D: Distributional Non-Standard
Model}). Then:
\begin{description}
\item[(i)] If $(A_\varphi)\in\mathbb{C}^{\mathcal{D}_0}$, then
$(A_\varphi)\in\mathcal{M}(\mathbb{C}^{\mathcal{D}_0})$ (Definition~\ref{D: Asymptotic
Numbers}) \ifff 
 $\left<A_\varphi\right>\in\mathcal{M}_\rho(^*\mathbb{C})$. 

\item[(ii)] The fields $\widehat{\mathbb{C}^{\mathcal{D}_0}}$ and
$\widehat{\mathbb{R}^{\mathcal{D}_0}}$ are isomorphic to
${^\rho\mathbb{C}}$ and ${^\rho\mathbb{R}}$, respectively, under
the mapping $\widehat{A_\varphi}\to \widehat{\left<A_\varphi\right>}$ from
$\widehat{\mathbb{C}^{\mathcal{D}_0}}$ to ${^\rho\mathbb{C}}$. This isomorphism preserves also the
valuation, non-archimedean norm and ultra-metric (Definition~\ref{D: Topology, Valuation, Ultra-Norm,
Ultra-Metric}).
\item[(iii)] The order topology and the metric topology on $\widehat{\mathbb{C}^{\mathcal{D}_0}}$ are the
same.
\end{description}
\end{theorem}
\begin{proof}  (i) $(A_\varphi)\in\mathcal{M}(\mathbb{C}^{\mathcal{D}_0})$ \ifff 
$(\exists m\in\mathbb{N})\{\varphi\in\mathcal{D}_0 : |A_\varphi|\leq (R_\varphi)^{-m}\}\in\mathcal{U}$
\ifff $(\exists m\in\mathbb{N})( |\left<A_\varphi\right>|\leq\rho^{-m})$ \ifff
$\left<A_\varphi\right>\in\mathcal{M}_\rho(^*\mathbb{C})$ as required.

	(ii) $\widehat{\left<A_\varphi\right>}=0$ in $^\rho\mathbb{C}$ \ifff
$(\forall n\in\mathbb{N})(|\left<A_\varphi\right>|<\left<R_\varphi\right>^n)$ in $^*\mathbb{C}$
\ifff  
\[
(\forall n\in\mathbb{N})(\left\{\varphi\in\mathcal{D}_0 : |A_\varphi|<
(R_\varphi)^n\right\}\in\mathcal{U}),
\]
\ifff
$(A_\varphi)\in\mathcal{N}(\mathbb{C}^{\mathcal{D}_0})$ (Definition~\ref{D: Asymptotic Numbers})
\ifff $\widehat{A_\varphi}=0$ in $\widehat{\mathbb{C}^{\mathcal{D}_0}}$ which means that the
mapping $\widehat{A_\varphi}\to\widehat{\left<A_\varphi\right>}$ is injective.  We leave to the
reader to verify that this mapping preserves the ring operations.

	(iii) The order topology and the metric topology on $\widehat{\mathbb{C}^{\mathcal{D}_0}}$ are the
same because they are the same on $^\rho\mathbb{C}$ for any choice of $^*\mathbb{C}$ and $\rho$
(Todorov \& Wolf~\cite{TodWolf}).
\end{proof} 

	  In what follows we assume a particular case of the generalized continuum hypothesis  in the form
$\mathfrak{c}^+=2^\mathfrak{c}$.

\begin{corollary}\label{C: A Generalization}  Let $^*\mathcal{R}$ be a non-standard extension of
$\mathbb{R}$ in a $\mathfrak{c}^+$-saturated non-standard model with set of individuals $\mathbb{R}$ such
that $\card(^*\mathcal{R})=\mathfrak{c}^+$. Let $\varepsilon$ be a positive infinitesimal in
$^*\mathcal{R}$ and let
${^\varepsilon\mathcal{C}}$ and
${^\varepsilon\mathcal{R}}$ be the corresponding Robinson's fields (see above).  Then 
${^\varepsilon\mathcal{C}}$ and  ${^\varepsilon\mathcal{R}}$ are isomorphic to
$\widehat{\mathbb{C}^{\mathcal{D}_0}}$ and $\widehat{\mathbb{R}^{\mathcal{D}_0}}$, respectively.
\end{corollary}
\begin{proof}   Let $^*\mathbb{R}$ be the non-standard extension of $\mathbb{R}$ in our distributional
non-standard model and let
$\rho=\left<R_\varphi\right>$ (\#\ref{No: Canonical Infinitesimal} in Definition~\ref{D: Distributional Non-Standard
Model}). We observe that 
$^*\mathbb{R}$ is fully saturated by  (Theorem~\ref{T: Saturation Principle}) and 
$^*\mathcal{R}$ is fully saturated by assumption.  Thus $^{\rho}\mathbb{R}$ and 
$^{\varepsilon}\mathcal{R}$ are isomorphic by (Todorov \& Wolf~\cite{TodWolf},
p.370). It follows that ${^\varepsilon\mathcal{R}}$ and
$\widehat{\mathbb{R}^{\mathcal{D}_0}}$ are isomorphic (as required) since  ${^\rho\mathbb{R}}$ and
$\widehat{\mathbb{R}^{\mathcal{D}_0}}$ are isomorphic by Theorem~\ref{T: Isomorphic
Fields}.
\end{proof} 

	 The sets of the form $B=\{z\in{^\rho\mathbb{C}}: |z-a|_v\leq b\}$, where 
$a\in{^\rho\mathbb{C}}$ and $b\in{\mathbb{R}_+}$, are called {\bf closed balls} in
$^\rho\mathbb{C}$. Similarly, if
$a\in{^\rho\mathbb{R}}$ and $b\in\mathbb{R}_+$, then the sets $B=\{z\in{^\rho\mathbb{R}}:
|z-a|_v\leq b\}$ are closed balls in ${^\rho\mathbb{R}}$. The next result is due to W.A.J.
Luxemburg~(\cite{wLux}, p.195). 
\begin{theorem}[Luxemburg] \label{T: Luxemburg} The field $^\rho\mathbb{R}$ is
\textbf{spherically complete}  in the sense that every family of closed balls in $^\rho\mathbb{R}$ with the
finite intersection property (f.i.p.) has non-empty intersection. Consequently,  the field
$^\rho\mathbb{C}$ is also {\em spherically complete}. 
\end{theorem}

	 We recall the definition of the algebra $^\rho\mathcal{E}(\Omega)$ (Oberguggenberger \&
Todorov\cite{OberTod98}).

\begin{definition}[$\rho$-Asymptotic Functions]\label{D: rho-Asymptotic Functions}{\em  Let
${^*\mathbb{R}}$ and ${^*\mathcal{E}(\Omega)}$ be the non-standard extensions of ${\mathbb{R}}$ and
$\mathcal{E}(\Omega)=:\mathcal{C}^\infty(\Omega)$, respectively, in an arbitrary
$\kappa$-saturated non-standard model with set of individuals $\mathbb{R}$, where $\kappa$ is an infinite
cardinal. (In particular, this could be the distributional non-standard model constructed in Section~\ref{S:
Distributional Non-Standard Model}). Let $\rho$ be positive infinitesimal in ${^*\mathbb{R}}$. 
Following (Oberguggenberger {\em \&} Todorov\cite{OberTod98}), we define:
\begin{enumerate}
\item  The sets of the \textbf{$\rho$-moderate} and
\textbf{$\rho$-negligible} functions in ${^*\mathcal{E}(\Omega)}$ are
\begin{align}
&\mathcal{M}_\rho(^*\mathcal{E}(\Omega))=\left\{f\in{^*\mathcal{E}(\Omega)}
: (\forall\alpha\in\mathbb{N}_0^d)(\forall
x\in\mu(\Omega))\left(\partial^\alpha
f(x)\in\mathcal{M}_\rho(^*\mathbb{C})\right)\right\},\notag\\
&\mathcal{N}_\rho(^*\mathcal{E}(\Omega))=\left\{f\in{^*\mathcal{E}(\Omega)}
: (\forall\alpha\in\mathbb{N}_0^d)(\forall
x\in\mu(\Omega))\left(\partial^\alpha
f(x)\in\mathcal{N}_\rho(^*\mathbb{C})\right)\right\},\notag
\end{align}
respectively, where $\mu(\Omega)$ is the monad of $\Omega$ (\#\ref{No: Monad} in Definition~\ref{D:
Distributional Non-Standard Model}). The differential algebra of
\textbf{$\rho$-asymptotic functions} on $\Omega$ is the factor ring
$^\rho\!\mathcal{E}(\Omega)=:
\mathcal{M}_\rho(^*\mathcal{E}(\Omega))/\mathcal{N}_\rho(^*\mathcal{E}(\Omega))$. We denote by
$\widehat{f}$ the equivalence class of
$f\in{\mathcal{M}_\rho(^*\mathcal{E}(\Omega))}$.  
\item For any
$\mathcal{S}\subseteq \,\,^{*}\mathcal{E}(\Omega)$ we let
$\widehat{\mathcal{S}} = \{\widehat{f} : f\in \mathcal{S}\cap\mathcal{M}_\rho(^*\mathcal{E}(\Omega))\}$.
If $S \subseteq \mathcal{E}(\Omega)$, 
the set $^{\rho}S=\widehat{^*S}$ is called the $\rho$-extension of $S$. The algebra  
$^\rho\mathcal{E}(\Omega)$ consists of particular pointwise
functions from $\widehat{\mu(\Omega)}$ into $^\rho\mathbb{C}$ (Todorov~\cite{tTod99}).
\end{enumerate}
}\end{definition}

	 The next result appears in (Oberguggenberger \& Todorov\cite{OberTod98}).
\begin{theorem}[Existence of Embedding]\label{T: Existence of Embedding} There exists an 
embedding $\Sigma_{D,\Omega}: \mathcal{D}^\prime(\Omega)\to{^\rho\!\mathcal{E}(\Omega)}$ of
Colombeau type, where $D\in{^*\mathcal{E}(\mathbb{R}}^d)$ stands for a particular 
non-standard delta-function (non-standard mollifier). Thus ${^\rho\!\mathcal{E}(\Omega)}$ are \textbf{special
algebras of generalized functions} of Colombeau's type (see the Introduction). 
\end{theorem}

\begin{remark} [Non-Canonical Embedding]\label{R: Non-Canonical Embedding} We should notice that
the embedding $\Sigma_{D,\Omega}$ is {\em non-canonical} because the existence of $D$ is proved in
\cite{OberTod98}  by saturation principle and thus $D$ cannot be
defined uniquely in the terms already used in the definition of ${^\rho\!\mathcal{E}(\Omega)}$. Actually,
$D$  cannot be determined uniquely by any properties expressed in the language of standard or
non-standard analysis; $D$ is chosen and fixed in \cite{OberTod98}  ``by hand''.
\end{remark}
	 The next two simple lemmas provide examples of the ability of non-standard analysis to reduce the
number of quantifiers.

\begin{lemma} Let $f\in{^*\mathcal{E}(\Omega)}$. Then the following are equivalent:
\begin{description}
\item[(a)]  $(\forall K\Subset\Omega)(\exists
m\in\mathbb{N})(\sup_{\xi\in{^*K}}|f(\xi)|\leq\rho^{-m})$.
\item[(b)] $(\forall \xi\in\mu(\Omega))(f(\xi)\in\mathcal{M}_\rho(^*\mathbb{C}))$.
\end{description}
\end{lemma}
\begin{proof}  (a) $\Rightarrow$ (b): Suppose that $\xi\in\mu(\Omega)$ and let $\st(\xi)=s$. Since
$s\in\Omega$ and
$\Omega$ is open, there exists an open relatively compact subset $\mathcal{O}$ of $\Omega$ which contains 
$s$ and such that $\overline{\mathcal{O}}\subset\Omega$. So, we have $\xi\in{^*K}$, where
$K=\overline{\mathcal{O}}$. Thus
$\sup_{\eta\in{^*K}}|f(\eta)|\leq\rho^{-m}$ for some $m\in\mathbb{N}$ (by assumption) implying
$f(\xi)\in\mathcal{M}_\rho(^*\mathbb{C})$ as required. 

	 (a) $\Leftarrow$ (b):  Let $K$ be a compact subset of $\Omega$ and suppose (on the
contrary) that 
 $(\forall m\in\mathbb{N})(\sup_{\xi\in{^*K}}|f(\xi)|>\rho^{-m})$. Next, we observe that the set
$\mathcal{A}=: \{m\in{^*\mathbb{N}} : \sup_{\xi\in{^*K}}|f(\xi)|>\rho^{-m}\}$ is internal by internal definition principle (Theorem~\ref{T: Internal Definition Principle}). Also, $\mathcal{A}$ 
contains $\mathbb{N}$ by assumption and thus $\mathcal{A}$ contains an infinitely large number $\nu\in{^*\mathbb{N}}$ by overflow of
$\mathcal{F}(^*\mathbb{R})$ (cf. Theorem~\ref{T: Spilling Principles} in this article or Capi\'{n}ski \& Cutland~\cite{CapinskiCutland95}, p.24). Thus we have
$\sup_{\xi\in{^*K}}|f(\xi)|>\rho^{-\nu}$. On the other hand, we have
$|f(\xi_0)|>\rho^{-\nu}$ for some $\xi_0\in{^*K}$ by transfer principle (cf. Theorem~\ref{T: Transfer Principle} in this article or Davis~\cite{mDavis}),
 contradicting (a), since ${^*K}\subset\mu(\Omega)$ by Theorem~\ref{T: Usual Topology on Rd and Monads}.
\end{proof}
\begin{lemma} Let $f\in{^*\mathcal{E}(\Omega)}$. Then the following are equivalent:
\begin{description}
\item[(a)]  $(\forall K\Subset\Omega)(\forall
n\in\mathbb{N})(\sup_{\xi\in{^*K}}|f(\xi)|<\rho^{n})$.
\item[(b)] $(\forall \xi\in\mu(\Omega))(f(\xi)\in\mathcal{N}_\rho(^*\mathbb{C}))$.
\end{description}
\end{lemma}
\begin{proof} The proof is very similar to the proof of the above lemma and we leave it to the reader.
\end{proof} 
\begin{theorem}[Isomorphic Algebras]\label{T: Isomorphic Algebras} Let $^*\mathcal{E}(\Omega)$ be the
non-standard extension of $\mathcal{E}(\Omega)$ (\#\ref{No: Extension Mapping},
Definition~\ref{D: Distributional Non-Standard Model}) in the distributional 
ultrapower non-standard  model constructed in Section~\ref{S: Distributional Non-Standard Model}. Let
$\rho=\left<R_\varphi\right>$ be the canonical infinitesimal  in
${^*\mathbb{R}}$ (\#\ref{No: Canonical Infinitesimal} in Definition~\ref{D: Distributional Non-Standard
Model}). Then:
\begin{description}
\item[(i)] If
$(f_\varphi)\in\mathcal{E}(\Omega)^{\mathcal{D}_0}$, then
$(f_\varphi)\in\mathcal{M}(\mathcal{E}(\Omega)^{\mathcal{D}_0})$ (Definition~\ref{D: Asymptotic
Functions}) \ifff 
 $\left<f_\varphi\right>\in\mathcal{M}_\rho(^*\mathcal{E}(\Omega))$.
\item[(ii)] The differential algebras $\widehat{\mathcal{E}(\Omega)^{\mathcal{D}_0}}$ and 
${^\rho\!\mathcal{E}(\Omega)}$ are isomorphic under
the mapping $\widehat{f_\varphi}\to \widehat{\left<f_\varphi\right>}$ from
$\widehat{\mathcal{E}(\Omega)^{\mathcal{D}_0}}$ to ${^\rho\!\mathcal{E}(\Omega)}$.
\end{description}
\end{theorem}
\begin{proof}  In view of the previous two lemmas, the proof of this theorem is almost identical to the proof
of Theorem~\ref{T: Isomorphic Fields} and we leave it to the reader. 
\end{proof} 

\section{The Hahn-Banach Extension Principle for Asymptotic Functionals}\label{S: A Hahn-Banach Like
Theorem for Asymptotic Functionals}

 \quad In this section we show that a Hahn-Banach extension principle holds for continuous asymptotic
functionals, i.e. linear continuous functionals defined on vector spaces over the field
$^\rho\mathbb{C}$ taking values also in $^\rho\mathbb{C}$ (Corollary~\ref{C: The Case
K=rhoC}). This result is based on the spherical
completeness of $^\rho\mathbb{C}$ (Luxemburg~\cite{wLux}, p.195 or Theorem~\ref{T:
Luxemburg} in this article) and a result due to (Ingleton~\cite{wIngleton}). Here $^\rho\mathbb{C}$ is
Robinson's field  (Definition~\ref{D: A. Robinson's rho-Asymptotic Numbers}) within an arbitrary non-standard
model with individuals $\mathbb{R}$ and $\rho$ is an arbitrary positive infinitesimal in $^*\mathbb{R}$.  
Consequently, the results in this section hold as well for linear continuos functionals with values in
$\widehat{\mathbb{C}^{\mathcal{D}_0}}$ (Definition~\ref{D: Asymptotic Numbers}) since
$\widehat{\mathbb{C}^{\mathcal{D}_0}}$ is isomorphic to  a field of the form $^\rho\mathbb{C}$
(Theorem~\ref{T: Isomorphic Fields}).

	 The rings of Colombeau generalized numbers $\overline{\mathbb{R}}$ and $\overline{\mathbb{C}}$
(Colombeau~\cite{jfCol84a}, pp.136) are also spherically complete, and a result similar to
Theorem~\ref{T: Hahn-Banach} appears in E. Mayerhofer's thesis~\cite{eMayerhofer}, where
$\mathbb{K}$ (see below) is a field which is a (proper) subring of $\overline{\mathbb{C}}$.  Also E.
Mayerhofer raised the question whether or not it is possible to generalize his result to the whole rings
$\overline{\mathbb{R}}$ and $\overline{\mathbb{C}}$ (cf. Conjecture 3.11
in Mayerhofer~\cite{eMayerhofer}). Later H. Vernaeve~\cite{hVernaeve} proved that a such
generalization is impossible. Thus Corollary~\ref{C: The Case K=rhoC} at the end of this section does not
have a  counterpart in  Colombeau theory. We look upon this fact as one more piece of evidence
supporting the point (advocated for a long time by the first author of this article) that  Robinson's field
$^\rho\mathbb{C}$ along with the algebra of asymptotic functions $^\rho\mathcal{E}(\Omega)$ are
better alternatives to  the ring of Colombeau's generalized scalars
$\overline{\mathbb{C}}$ and Colombeau's algebra of generalized functions $\mathcal{G}(\Omega)$ for the
purpose of non-linear theory of generalized functions and functional analysis in general.

 	The reader might observe some similarity between the field 
$^\rho\mathbb{R}$ (and $^\rho\mathbb{C}$ as well) and the fields of the $p$-adic numbers
$\mathbb{Q}_p$ (Ingleton~\cite{wIngleton}). This similarity is due to the fact
that $^\rho\mathbb{R}$,  $^\rho\mathbb{C}$ and 
$\mathbb{Q}_p$ are all ultra-metric spaces. For a discussion on this topic we
refer to (Luxemburg~\cite{wLux}). We should mention, however, that the fields
$^\rho\mathbb{R}$,  $^\rho\mathbb{C}$ and 
$\mathbb{Q}_p$ are quite different from each other. For example, each
$^\rho\mathbb{R}$ (just like 
$^*\mathbb{R}$) is a real closed, and thus, a totally ordered field. Also each
$^\rho\mathbb{C}$ (just like  $^*\mathbb{C}$) is a algebraically closed field. In contrast, the fields
$\mathbb{Q}_p$ are neither algebraically closed, nor real closed fields. In fact  $\mathbb{Q}_p$ are not even
real fields, that is to say that $\mathbb{Q}_p$ are non-orderable. Recall that a field
$\mathbb{K}$ is {\em orderable} \ifff 
$\mathbb{K}$ is  {\em real} in the sense that equations of the form
$x_1^2+x_2^2+\dots + x_n^2=0$ admit only trivial solutions  $x_1=x_2=\dots= x_n=0$ in $\mathbb{K}$
(Van Der Waerden~\cite{VanDerWaerden}, Chapter 11). Neither of the fields $\mathbb{Q}_p$  has this
property (Ribenboim~\cite{pRib}, pp.144-145). 

		 We start with some preliminaries: 

\begin{enumerate}
\item Let $\mathbb{K}$ be a subfield of $^\rho\mathbb{C}$. Let $V$ be a  vector space over $\mathbb{K}$
and let
$||\cdot||_v: V\to\mathbb{R}$ be a ultra-norm on $V$. The latter means that for every 
$x, y\in V$ and $c\in\mathbb{K}$, we have: {\bf (a)} $||x||_v\geq 0$ and $||x||_v=0$ occurs only if $x=0$;
{\bf (b)} $||c x||_v=
|c|_v\, ||x||_v$, where $|c|_v$ is defined in \#\ref{No: Ultra-Norm} in Definition~\ref{D: A. Robinson's
rho-Asymptotic Numbers}; {\bf (c)}
$ ||x+y||_v\leq\max\{||x||_v, ||y||_v\}$ (ultra-norm inequality). We denote by
$(V,
\mathbb{K}, ||\cdot ||_v)$ the corresponding {\bf ultra-normed vector space} over $\mathbb{K}$. Notice, in
particular, that if 
$V$ is an inner vector space over $\mathbb{K}$, then the formula
$||x||_v=\sqrt{|(x, x)|_v}$ defines a ultra-norm on $V$. Also, if $\mathbb{K}$ is an algebraically closed (or real
closed) field, then the formula $||x||_v=|\sqrt{(x, x)}|_v$ also produces a ultra-norm on $V$.

\item Let $V^*$ be the \textbf{algebraic dual} of $V$, i.e. the vector space over $\mathbb{K}$ of all
linear functionals $T: V\to {\mathbb{K}}$. We shall use the same notation, $||\cdot||_v$,  for the 
non-archimedean norm $||\cdot||_v: V^*\to\mathbb{R}\cup\{\infty\}$ inherited from $V$ by duality, i.e.
$||F||_v =\sup_{\overset{x\in V}{||x||_v=1}}\; |T(x)|_v$.
\item  $T\in V^*$ is called \textbf{continuous} if $||T||_v\in\mathbb{R}$ (i.e. if $||T||_v\not=\infty$). We
denote by
$V^\prime$ the vector space over $\mathbb{K}$ of all continuous functionals in $V^*$. 
Thus  $|T(x)|_v\leq ||T||_v\; ||x||_v\in\mathbb{R}$ holds for all $T\in V^\prime$ and all $x\in V$.
\end{enumerate}
	
	  Here is our Hahn-Banach extension principle. 
\begin{theorem}[Hahn-Banach]\label{T: Hahn-Banach}
Let
$\mathbb{K}$ be a subfield of
$^\rho\mathbb{C}$ which is spherically complete under the ultrametric on $^\rho\mathbb{C}$. Let 
$(V, \mathbb{K}, ||\cdot ||_v)$ be an ultra-normed vector
space over $\mathbb{K}$. Let
$U$ be a $\mathbb{K}$-linear subspace of
$V$. Then every functional 
$T\in U^\prime$ can be extended (non-uniquely) to a functional $M\in V^\prime$ such that
$||T||_v=||M||_v$.
\end{theorem} 
\begin{proof} The above theorem is a particular case of A. W. Ingleton's result in \cite{wIngleton}. A
similar independent proof follows: If $U=V$ there is nothing to prove. Suppose
$x_0\in V\setminus U$ and let
$U+\mathbb{K}\, x_0=\{x+c x_0 : x\in U, c\in\mathbb{K}\}$. Our first goal is to extend $T$ to a
functional
$S$ on $U+\mathbb{K}\, x_0$. We let
$S(x+c x_0)= T(x)+c S(x_0)$. To complete this
definition we need to prescribe a value $S(x_0)=y_0$ which preserves the $v$-norm of
$T$. For any $x\in U$ we define $R(x)=:||T||_v\; ||x-x_0||_v$ and the closed ball $\mathcal{B}_x=
\{y\in{\mathbb{K}} : |y-T(x)|_v\leq R(x)\}$.
The family $\{\mathcal{B}_x\}_{x\in U}$ has the finite intersection property. Indeed, suppose that
$x_1, x_2\in U$ and $x_1\not= x_2$. By involving the ultra-norm
inequality, we obtain: $ |T(x_1)-T(x_2)|_v\leq ||T||_v\; ||x_1-x_2||_v\leq ||T||_v\; \max\{ || x_1-x_0||_v, \, ||
x_0-x_2||_v\}=\max\{R(x_1),\, R(x_2)\}$. If $R(x_1)\leq R(x_2)$, then 
$T(x_1)\in\mathcal{B}_{x_2}$. If  $R(x_1)\geq R(x_2)$, then 
$T(x_2)\in\mathcal{B}_{x_1}$.  Since $T(x_1)$ and 
$T(x_2)$ are the centers of the balls $\mathcal{B}_{x_1}$ and $\mathcal{B}_{x_2}$,
respectively, it follows that either $T(x_1)\in\mathcal{B}_{x_1}\cap\mathcal{B}_{x_2}$, or
$T(x)\in\mathcal{B}_{x_1}\cap\mathcal{B}_{x_2}$. In either case we have
$\mathcal{B}_{x_1}\cap\mathcal{B}_{x_2}\not=\varnothing$, as required. Notice that the latter
implies either $\mathcal{B}_{x_1}\subset\mathcal{B}_{x_2}$, or
$\mathcal{B}_{x_1}\supset\mathcal{B}_{x_2}$ due to  the ultra-norm inequality, hence the argument can be
repeated for any finite number of elements in $U$. Thus there exists
$y_0\in\bigcap_{x\in U}\mathcal{B}_x$ since
$\mathbb{K}$ is spherically complete by assumption. We let
$S(x_0)=y_0$ and the definition of
$S$ is complete. Clearly $S$ is an extension of $T$. To
show the preservation of the norm of $T$, observe that $||T||_v\leq||S||_v$ trivially. To show
$||T||_v\geq||S||_v$, suppose that $x-c x_0\in U+\mathbb{K}\, x_0$.  Notice that 
$|y_0-T(x)|_v\leq R_x$ for all $x\in U$, by the choice of $y_0$. 
With this in mind, we estimate
$|S(x-cx_0)|_v=|c|_v\,|S(x/c-x_0)|_v=|c|_v\, |S(x/c)-S(x_0)|_v=|c|_v\, |T(x/c)-y_0|_v\leq|c|_v\, R(x/c)=
|c|_v\, ||T||_v\, ||x/c-x_0||_v=||T||_v\, ||x-cx_0||_v$. Thus $||T||_v=||S||_v$, as required. The rest of the proof is a
typical application of Zorn's lemma. Let
$\mathcal{L}_T$ denote the set of all extensions of $T$ which preserves the $v$-norm of $T$. Notice that
$\mathcal{L}_T\not=\varnothing$ since $S\in\mathcal{L}_T$. If 
$T_1, T_2\in\mathcal{L}_T$ and  if $T_2$ is an extension of $T_1$, we shall write $T_1\prec T_2$ and also
$T_1\vee T_2=T_2$. Let $\mathcal{L}$ be a subset of $\mathcal{L}_T$ which is totally ordered
under $\prec$. We observe that $\bigvee_{L\in\mathcal{L}} L\in\mathcal{L}_T$. Thus,
by Zorn's lemma, $\mathcal{L}_T$ has maximal elements; let $M$ be a such maximal element of
$\mathcal{L}_T$. The functional $M$ is an extension of $T$ and $||T||_v = ||M||_v$, by the definition of
$\mathcal{L}_T$.  The domain of $M$ should be $V$. Suppose (on the contrary) that the domain ${\rm
dom}(M)$ of $M$ is a proper subset of $V$. Then we can, as before, choose $z_0\in V\setminus \dom(M)$
and extend $M$ to $\dom(M)+\mathbb{K}z_0$, contradicting the maximality of $M$. 
\end{proof} 

\begin{example}[Power Series] Let  $\mathbb{C}\left<x\right>$ be the Levi-Civita field consisting of all
formal series of the form $\sum_{n=0}^\infty a_n x^{r_n}$, where $a_n\in\mathbb{C}$ and  $(r_n)$ is a
strictly increasing unbounded sequence in
$\mathbb{R}$ (Levi-Civita~\cite{tLC}). The field $\mathbb{C}\left<x\right>$ is isomorphic to the {\em field
of algebraic functions in one variable} in the sense that  $\mathbb{C}\left<x\right>$ is an algebraic
closure of the field of rational functions $\mathbb{C}(x)$. The field
$\mathbb{C}\left<x\right>$ is spherically complete  (Luxemburg~\cite{wLux}) and it can be
embedded in $^\rho\mathbb{C}$ by the mapping $\sum_{n=0}^\infty a_n
x^{r_n}\to\sum_{n=0}^\infty a_n
\rho^{r_n}$ (cf. Robinson~\cite{aRob73} or Lightstone \& Robinson~\cite{LiRob}). The above
Hahn-Banach extension principle holds for its image
$\mathbb{K}=\mathbb{C}\left<\rho\right>$.  For more examples of spherically complete algebraically
closed and real closed subfields $\mathbb{K}$ of 
$^\rho\mathbb{C}$, we refer to  (Todorov \& Wolf~\cite{TodWolf}).
\end{example}

 The next result does not have a counterpart in  Colombeau theory (Vernaeve~\cite{hVernaeve}) since
$\overline{\mathbb{R}}$ and $\overline{\mathbb{C}}$  are rings with zero  divisors.

\begin{corollary}[The Case $\mathbb{K}={^\rho\mathbb{C}}$]\label{C: The Case K=rhoC} Let 
$(V,\,  {^\rho\mathbb{C}},\,  ||\cdot ||_v)$ be a ultra-normed vector space over the field $^\rho\mathbb{C}$.
Let $U$ be a $^\rho\mathbb{C}$-linear subspace of\, $V$. Then every functional 
$T\in U^\prime$ can be extended (non-uniquely) to a functional $M\in V^\prime$ such that $||T||_v=||M||_v$. A
similar result holds about any ultra-normed vector space $(V,\,  {^\rho\mathbb{R}},\,  ||\cdot ||_v)$ over the
field $^\rho\mathbb{R}$.
\end{corollary}

\begin{proof}  Since both $^\rho\mathbb{C}$ and $^\rho\mathbb{R}$ are spherically complete fields 
 (cf. Luxemburg [22], p. 195 or Theorem 7.6 in this article), we can apply the above theorem for 
$\mathbb{K}={^\rho\mathbb{C}}$ and
$\mathbb{K}={^\rho\mathbb{R}}$, respectively.  
\end{proof} 

{\bf Acknowledgment:} The authors thank Michael Oberguggenberger for the useful discussion on
preliminary versions of this article. The first author is grateful to the colleagues at the 
University of Vienna and especially to Michael Kunzinger for the hospitality and financial
support in the Spring of 2006, where the work on this article began.

\end{document}